\documentclass[11pt]{amsart}
\usepackage[left=3cm,right=3cm,top=2cm,bottom=2cm]{geometry} 

\usepackage{amsmath,amssymb,amscd,amsthm,amsxtra}
\usepackage{dsfont}
\usepackage{graphicx, mathrsfs}
\usepackage{amssymb,amsmath,amsthm}
\usepackage{mathrsfs,dsfont}
\usepackage{fancyhdr}
\usepackage{ulem}
\usepackage{enumerate}
\usepackage{hyperref}
\usepackage[toc,page]{appendix}
\usepackage{color}

\newtheorem{thm}{Theorem}[section]
\newtheorem{prop}[thm]{Proposition}
\newtheorem{lem}[thm]{Lemma}
\newtheorem{cor}[thm]{Corollary}
 
\theoremstyle{definition}
\newtheorem{definition}[thm]{Definition}

\theoremstyle{remark}

\newtheorem{remark}[thm]{Remark}
\numberwithin{equation}{section}

\newcommand{\ZZZ}{\mathbb{Z}}
\newcommand{\RRR}{\mathbb{R}}

\newcommand{\TTT}{\mathbb{T}}
\newcommand{\PPP}{\mathbb{P}}
\newcommand{\EEE}{\mathbb{E}}
\newcommand{\N}{\mathcal{N}}
\newcommand{\J}{\mathcal{J}}
\newcommand{\dd}{\delta}
\newcommand{\w}{\omega}
\newcommand{\case}[2]{ \noindent \textbf{Case #1:} #2 \\}

\begin{document}


\title{
Almost sure well-posedness for the cubic nonlinear Schr\"{o}dinger equation in the super-critical regime on $\TTT^d$, $d\geq 3$}


\author{HAITIAN YUE}
\address{Department of Mathematics and Statistics, University of Massachusetts Amherst}
\email{hyue@math.umass.edu}



\begin{abstract}
In this paper we prove almost sure local well-posedness in both atomic spaces $X^s$ and Fourier restriction spaces $X^{s, b}$ for the cubic nonlinear Schr\"odinger equation on $\TTT^d$ ($d\geq 3$) in the super-critical regime.
\end{abstract}

\maketitle


\section{Introduction}

We consider the Cauchy initial value problem for the cubic nonlinear Schr\"odinger equation (NLS) in the d-dimensional tori $\TTT^d$ ($d\geq 3$)

\begin{equation}\label{NLS.IVP}
  \begin{cases}
  iu_t + \Delta u = \rho u|u|^{2},  \qquad \rho = \pm 1, \qquad x\in \mathbb{T}^d\ (d\geq 3)\\
  u(0, x) = \phi^\omega(x).\\
  \end{cases}
\end{equation}

The initial data $\phi^\omega(x)$ in (\ref{NLS.IVP}) is defined by randomization.
\begin{equation}\label{RandomInitialData}
\phi^{\omega}(x) = \sum_{n\in \mathbb{Z}^d} \frac{g_n(\omega)}{\langle n\rangle^{d-1-\alpha}} e^{in\cdot x},\text{ where } \langle n\rangle = \sqrt{1+|n|^2},
\end{equation}
 where $(g_n(\omega))_{n\in \mathbb{Z}^d}$ is a sequence of complex i.i.d. mean zero Gaussian random variables on a probability space $(\Omega, A, \mathbb{P})$.

\begin{remark}
Let's consider a function $\phi \in H^{s_c-\alpha-\epsilon}(\mathbb{T}^d)$ for any $\epsilon > 0$ of the form
\begin{equation}\label{InitialData}
\phi(x) = \sum_{n\in \mathbb{Z}^d} \frac{1}{\langle n\rangle^{d-1-\alpha}} e^{in\cdot x}.
\end{equation}
If we replace the Fourier coefficients of (\ref{InitialData}) with randomized coefficients $\frac{g_n(\omega)}{\langle n\rangle^{d-1-\alpha}}$, then the randomization of (\ref{InitialData}) becomes the {\it random initial data} (\ref{RandomInitialData}) of (\ref{NLS.IVP}). It's easy to see that $\phi^{{\omega}}(x)$ is a.s. in $H^{s_c-\alpha-\epsilon}$, {but not} in $H^s$, $s\geq s_c-\alpha$. Thus randomization does not regularize the data
in the scale of the Sobolev spaces.
\end{remark}

In the Euclidean space $\RRR^d$, the scaling symmetry plays an important role on the well-posedness (existence, uniqueness and continuous dependence of the data to solution map) theory of the Cauchy initial value problem (IVP) for NLS:
\begin{equation}\label{eq:pNLS}
\begin{cases}
i\partial_t u + \Delta u = |{u}|^{p-1} u,\qquad p>1\\
u(0, x) = u_0(x) \in \dot{H}^s(\RRR^d).
\end{cases}
\end{equation}
The IVP (\ref{eq:pNLS}) is scaling invariant in the Sobolev norm $\dot{H}^{s_c}$, where $s_c := \frac{d}{2} - \frac{2}{p-1}$ is so-called scaling critical regularity. 
Initial data in  $\dot{H}^s$ with $s > s_c$ ({sub-critical regime}) is the best possible setting for well-posedness. Indeed, local-in-time well-posedness of (\ref{eq:pNLS}) was proven by Cazenave-Weissler in \cite{caz1}.   

For $\dot{H}^s$ data with $s = s_c$ ({critical regime}) the well-posedness problem is more difficult than the one in the sub-critical regime.  In fact, the well-posedness in the sub-critical regime can be obtained from the well-posedness in the critical regime by a persistence of regularity argument. Bourgain \cite{bourgain1999radial} first proved the large data global-in-time well-posedness and scattering for the defocusing energy-critical ($s_c = 1$) NLS in $\RRR^3$  with radially symmetric initial data in $\dot{H}^1$ by introducing an induction method on the size of energy and a refined Morawetz inequality. A different proof of the same result was given by Grillakis in \cite{grillakis2000radial}. A breakthrough was made by Colliander-Keel-Staffilani-Takaoka-Tao in \cite{colliander2008global}. Their work extended the results of Bourgain \cite{bourgain1999radial} and Grillakis \cite{grillakis2000radial}. They proved global-in-time well-posedness and scattering of the energy-critical problem in $\RRR^3$ for general large data in $\dot{H}^1$.
Similar results were then proven by Ryckman-Vi{\c{s}}an \cite{visan2007global} on the higher dimension $\RRR^d$ spaces. Furthermore, Dodson proved mass-critical ($s_c = 0$) global-in-time wellposedness results for $\RRR^d$ in his series of papers \cite{dodson2012global, dodson2016global1, dodson2016global2}.

Data in $\dot{H}^s$ with $s < s_c$ ({super-critical regime}) is rougher than 
the critical regularity data. 
Intuitively, in this case, scaling is 'against well-posedness'. This intuition was verified for example in  
\cite{christ2003ill}\cite{christ2003ill2}, where it is shown that super-critical data lead the initial value problem for NLS in $\RRR^d$ to ill-posedness. More precisely, they show that the solutions whose $\dot{H}^s$ norms become arbitrary large in arbitrary small time with arbitrary small initial data can be constructed. These solutions, exhibiting -what is called- {\it norm inflation}, contradict, in particular, the continuous dependence on the initial data.

However, ill-posedness in some cases can be 
circumvented by an appropriate probabilistic method in some probability space of initial 
data, in the other words, one may hope to establish almost sure LWP with respect to certain probability 
random data space. This random data approach to well-posedness first appeared in Bourgain's series 
 of papers \cite{bourgain1994periodic}\cite{bourgain1996periodic2d} in the context of studying the invariance of Gibbs 
measures associated to NLS on tori ($\TTT$ and $\TTT^2$). Later, Burq-Tzvetkov \cite{burq2008nlw1}\cite{burq2008nlw2} obtained similar 
results in the context of  the cubic
nonlinear wave equation (NLW) on a three dimensional compact Riemannian manifold.
{ The random data approach to wellposedness has also been pursued by many authors and applied to several nonlinear evolution equations on different manifolds ($\RRR^d$, $\TTT^4$ or $\mathbb{S}^d$ etc.)  to obtain almost sure local -and in some instances almost sure global- well-posedness results.} {Some references  in the context of NLS include:} \cite{nahmod2012dnls, CollianderOh2012nls1d, Deng2012nls2d, BourgainBulut1, BourgainBulut2, BourgainBulut3, nahmod2015almost, OanaNLSradial, HirayamadNLS2016,  killip2017almost, oh2017probabilistic, murphy2017random, dodson2018almost}; { in the context of  NLW include:} \cite{burq2008nlw1, burq2008nlw2, zhong2012nlw, burq2014, Luhrmann2014nlw, Luhrmann2016nlw, OanaNLW2016R3, xiabo2016, OanaNLW2017R45, dodson2017almost, bringmann2018almost};{ and in the context of  Navier-Stokes equations include:} \cite{DengCuiNS2011,  ZhangFangNS2012, nahmod2013NS, wang2017}. Recently Dodson-L\"uhrmann-Mendelson \cite{dodson2017almost} first established 
almost sure scattering for cubic NLW in $\RRR^4$ with randomized radially symmetric initial data in the super-critical regime. Then Killip-Murphy-Vi\c{s}an \cite{killip2017almost} and Dodson-L\"uhrmann-Mendelson \cite{dodson2018almost} proved similar almost sure scattering results with randomized radial data for cubic NLS on $\RRR^4$.

In this paper, we study the cubic NLS in the super-critical regime on tori $\TTT^d$ $(d\geq 3)$ via the probabilistic approach. After Bourgain's first two papers \cite{bourgain1994periodic}\cite{bourgain1996periodic2d} on $\TTT^1$ and $\TTT^2$, Nahmod-Staffilani \cite{nahmod2015almost} proved an almost sure local-in-time well-posedness result for the periodic 3D quintic NLS with an appropriate gauge transform in the super-critical regime.
This paper follows the similar spirit and obtain local-in-time well-posedness in high probability in the adapted atomic spaces $X^s$ by introducing a new lemma which modifies the "transfer principle" (Prop \ref{timelocalTransferPrinciple}) of atomic spaces up and focuses on the estimates in the small time intervals. In this paper, we construct a probability measure for the function space of initial data and show that the solutions exist for high probability of initial data.


Our main result can be stated as following:

\begin{thm}[Main Theorem] \label{MainThm}
Suppose $d\geq 3$ and   
\begin{equation}\label{coef:srd}
s_r(d) = \begin{cases}
\frac{1}{7}& d= 3\\
\frac{4}{19}& d= 4\\
\frac{1}{4} & d\geq 5.
\end{cases}
\end{equation}
Let $0\leq \alpha< s_r(d)$, $s\in [s_c, s_c + s_r(d)-\alpha)$.  Then there exists $\delta_0 > 0$ and $r = r(s, \alpha) > 0$ such that for any $ 0 < \delta < \delta_0$, there exists $\Omega_\delta \in A$ with
$$
\mathbb{P}(\Omega_\delta^c) < e^{-\frac{1}{\delta^r}},
$$
and for each $\omega \in \Omega_\delta$ there exists a unique solution $u$ of (\ref{NLS.IVP}) in the space
$$
S(t) \phi^{\omega} + X^s([0, \delta])_{\text{dist}},
$$
where $S(t)\phi^{\omega}$ is the linear evolution of the initial data $\phi^{\omega}$ given by (\ref{RandomInitialData}).
\end{thm}

Here we denoted by $X^s([0, \delta])_{\text{dist}}$ 
the metric space $(X^s([0, \delta]), {\text{dist}})$ where ${\text{dist}}$ is the metric defined by (\ref{Metric}) 
and $X^s([0, \delta])$ is the adapted atomic space introduced  in the Definition \ref{def:Xs}.

\begin{remark}
We also prove the analog of Main Theorem in $X^{s, b}$ (Theorem
\ref{thm:MainThmXsb}) instead of the atomic space $X^s$ in the Section 7, but we hold the theorem in $X^{s, b}$ only when $s\in (s_c, s_c+ s_r(d)-\alpha]$ (the proof of Theorem \ref{thm:MainThmXsb} fails when $s = s_c$). If we only consider the statement of theorems, for some $s>s_c$, the solution space
$
S(t) \phi^{\omega} + X^{s, b}([0, \delta])_{\text{dist}}
$  is indeed in the space $
S(t) \phi^{\omega} + X^{s_c, b}([0, \delta])_{\text{dist}}
$. However, the proof of $s= s_c$ case is still important in the sense that  we obtain the nonlinear estimate at the regularity of $s_c$. Especially in the case of $s_c = 1$, the nonlinear estimate at the regularity of $s_c$ would be necessary if we try to control the energy in a long-time term.

\end{remark}

To prove Theorem \ref{MainThm}, first we consider the initial value problem below,
\begin{equation}\label{GNLS.IVP}
  \begin{cases}
  iv_t + \Delta v = \mathcal{N}(v),  \qquad \rho = \pm 1, \qquad x\in \mathbb{T}^d\\
  v(0, x) = \phi^\omega(x),\\
  \end{cases}
\end{equation}
where 
\begin{equation}\label{GNonlinearity}
\mathcal{N}(v_1, v_2, v_3) := \rho (v_1 {v}_2 v_3 -  2 v_1 \int_{\mathbb{T}^d} {v}_2 v_3 dx) = \N_1(v_1, v_2, v_3) + \N_2(v_1, v_2, v_3),
\end{equation}
and set $\mathcal{N}(v) := \mathcal{N}(v, \overline{v}, v)$.

Suppose $\beta_v(t) = 2\int_{\mathbb{T}^d} |v|^2 dx$ and define $u(t,x) : = e^{-i\rho \beta_v(s) ds} v(t, x)$. We observe that u solve IVP (\ref{NLS.IVP}). Now suppose that one obtains well-posedness for the IVP (\ref{GNLS.IVP}) in a certain Banach space $(X, \|\cdot\|)$ then one can transfer those results to the IVP (\ref{NLS.IVP}) by using a metric space $X_{\text{dist}} : = (X, {\text{dist}})$
where
\begin{equation}
\label{Metric}
d(u, v) : = \| e^{i\rho \beta_u(s) ds} u(t, x) - e^{i\rho \beta_v(s) ds}v(t, x)\|.
\end{equation}

We define 
\begin{equation}\label{def:v0}
v_0^\omega = S(t) \phi^{\omega}(x),
\end{equation} 
and $w(x, t)$ solves the following the IVP (\ref{RNLS.IVP}), then we know that $v = v_0^\omega + w$ solves the IVP (\ref{GNLS.IVP}) which is the gauged NLS we want to solve.
\begin{equation}\label{RNLS.IVP}
  \begin{cases}
  iw_t + \Delta w = \mathcal{N}(w + v_0^\omega),   \qquad x\in \mathbb{T}^d\\
  w(0, x) = 0,\\
  \end{cases}
\end{equation}
where $\mathcal{N}(\cdot)$ was defined in (\ref{GNonlinearity}).

We are now ready to state the almost sure well-posedness result for the IVP (\ref{RNLS.IVP}) which implies the main theorem (Theorem \ref{MainThm}).

\begin{thm}\label{ASThm}
Suppose $d\geq 3$ and  $s_r(d)$ is defined as (\ref{coef:srd}). 
Let $0\leq \alpha< s_r(d)$, $s\in [s_c, s_c + s_r(d)-\alpha)$.  Then there exists $\delta_0 > 0$ and $r = r(s, \alpha) > 0$ such that for any $ 0 < \delta < \delta_0$, there exists $\Omega_\delta \in A$ with
$$
\mathbb{P}(\Omega_\delta^c) < e^{-\frac{1}{\delta^r}},
$$
and for each $\omega \in \Omega_\delta$ there exists a unique solution $w$ of (\ref{RNLS.IVP}) in the space
$ X^s([0, \delta])\cap C([0, \delta], H^s(\mathbb{T}^4))$.
\end{thm}

\subsection*{Outline of the following paper}
The rest of the paper is organized as follows. In Section 2, we state some basic probabilistic properties the proof depends on. In Section 3, we introduce the adapted atomic spaces $X^s$ and $Y^s$, provide some corresponding embedding properties of the spaces and furthermore obtain a transfer principle proposition (Proposition \ref{timelocalTransferPrinciple}) focusing on the small time intervals. 
Section 4 contains some Strichartz estimates, lattice counting lemmata and other lemmata we rely upon.
In Section 5, we estimate the nonlinear terms in the $X^s$-norm case by case. Section 6 contains statements on almost sure local well-posedness for the gauged Cauchy initial value problem (\ref{RNLS.IVP}) by using the nonlinear estimate in Section 5. In Section 7, we prove an analog result of almost sure local well-posedness  in $X^{s, b}$ spaces of the main theorem (Theorem \ref{MainThm}).

\subsection*{Acknowledgments}
The author is greatly indebted to his advisor, Andrea R. Nahmod, for suggesting this problem and her patient guidance and warm encouragement over the past years. The author also would like to thank Prof. Gigliola Staffilani for correcting one error in an earlier version of the paper and several helpful discussions in MSRI. The author acknowledges support from the National Science Foundation through his advisor Andrea R. Nahmod’s grants {NSF}-{DMS} 1201443 and {NSF}-{DMS} 1463714.

\section{Probabilistic set up}

\begin{lem}\label{Neps}
Let $\{g_n(\omega)\}_{n\in\mathbb{Z}^d}$ be a sequence of complex i.i.d. mean zero Gaussian random variables on a probability space $(\Omega, A, \mathbb{P})$. Then given $\epsilon, \delta > 0$, there exists a subset $\Omega_\dd\subset \Omega$ satisfying $\PPP(\Omega_\dd^c) \leq e^{-\frac{1}{\dd^{\epsilon}}}$, such that
$$
|g_n(\omega)| \lesssim \frac{1}{\delta^\epsilon} \log (\langle n\rangle+1).
$$
\end{lem}
\begin{proof}
For each $n$ and a small $\epsilon > 0$, we have a constant $C$, 
$$\mathbb{E}e^{|g_n(\omega)|} \leq C.$$  
Set $M=\frac{1}{\dd^\epsilon}$, and the we have 
$$
\mathbb{E} \left|\frac{e^{|g_n(\omega)|}}{e^M}\right|\leq Ce^{-\frac{1}{\dd^{\epsilon}}}
$$
Then we obtain,
$$
Ce^{-\frac{1}{\dd^{\epsilon}}} > \mathbb{E} \left|\frac{e^{|g_n(\omega)|}}{e^{M}}\right| \geq \sum_{j\in \ZZZ^d}\mathbb{P}(e^{|g_j(\omega)|}\geq e^{M} \langle j\rangle^{d})= \sum_{j\in \ZZZ^d}\mathbb{P}({|g_j(\omega)|}\geq \frac{1}{\dd^\epsilon} + d \log \langle j\rangle).
$$

Exclude $\Omega_\dd^c := \cup_{j} \{|g_j(\omega)|\geq \frac{1}{\dd^\epsilon} + d \log \langle j\rangle\}$ from $\Omega$, for all $\omega \in \Omega_\dd$, we have 
$$
|g_n(\omega)|\leq  \frac{1}{\dd^\epsilon} + d \log \langle n\rangle\lesssim \frac{1}{\delta^\epsilon} \log (\langle n\rangle+1), \text{for } n\in \ZZZ^d.
$$
with $\mathbb{P}(\Omega_\dd^c)<Ce^{-\frac{1}{\dd^{\epsilon}}}$.

\end{proof}

\begin{lem}[Lemma 3.1 in \cite{nahmod2015almost}]\label{LargeDeviation}
Let $\{g_n(\omega)\}_{n}$ be a sequence of complex i.i.d. mean zero Gaussian random variables on a probability space $(\Omega, A,\mathbb{P})$ and $(c_n) \in \ell^2$. Define
\begin{equation}
F(\omega) := \sum_{n} c_n g_n(\omega).
\end{equation}
Then there exists $C > 0$ such that for every $\lambda > 0$ we have
\begin{equation}
\mathbb{P}(\{\omega : |F(\omega)| > \lambda\}) \leq \exp(\frac{-C \lambda^2}{\|F(\omega)\|^2_{L^2(\Omega)}}).
\end{equation}
As a consequence there exists $C>0$ such that for every $q\geq 2$ and every $(c_n)\in \ell^2$,
$$
\|\sum_n c_n g_n(\omega)\|_{L^q(\Omega)} \leq C\sqrt{q} (\sum_n |c_n|^2)^{\frac{1}{2}}.
$$
\end{lem}

\begin{lem}[Lemma 3.5 in \cite{colliander2008global}]\label{Lem2}
Let $f^\omega (x,t) = \sum c_n g_n(\omega) e^{i(n\cdot x +|n|^2t|)}$. Then, for $p,q\geq 2$, there exists $\delta_0, c, C>0$ such that
\begin{equation}\label{ExpLarge}
\mathbb{P}(\|f^\omega\|_{L^p_tL^q_x(\mathbb{T}^4\times [0, \delta])}>\lambda) < C\exp{(-\frac{c\lambda^2}{\delta^{\frac{2}{p}}\|c_n\|^2_{l^2_n}})}
\end{equation}
for $\delta <\delta_0$.
\end{lem}
\begin{proof}
By Lemma \ref{LargeDeviation}, there exists  $C>0$ such that 
$$
\|\sum_n c_n g_n(\omega)\|_{L^r(\Omega)}\leq C\sqrt{r}(\sum_n |c_n|^2)^\frac{1}{2},
$$
for every $r\geq 2$. By Minkowski integral inequality, we have
\begin{align*}
\EEE(\|f^\omega\|_{{L^p_tL^q_x(\mathbb{T}^4\times [0, \delta])}}^r)^{\frac{1}{r}}
&\leq \big\|\|f^\omega\|_{L^r(\Omega)}\big\|_{{L^p_tL^q_x(\mathbb{T}^4\times [0, \delta])}}\\
&\leq C\sqrt{r} \big\|\|c_n\|_{l^2_n}\big\|_{{L^p_tL^q_x(\mathbb{T}^4\times [0, \delta])}}\\
&\leq C\sqrt{r} \delta^{\frac{1}{p}}\|c_n\|_{l^2_n}
\end{align*}
for $r\geq p$. By Chebyshev's Inequality, we have
\begin{equation}\label{LpLarge}
\PPP(\|f^\omega\|_{L^p_tL^q_x(\mathbb{T}^4\times [0, \delta])}>\lambda) < C^r \lambda^{-r} r^{\frac{r}{2}}\delta^{\frac{r}{p}} \|c_n\|_{l_n^2}^r.
\end{equation}

If $\lambda < \sqrt{p}C\delta^{-\frac{1}{p}}e\|c_n\|_{l^2_n}$, then (\ref{ExpLarge}) easily holds.

If  $\lambda\geq \sqrt{p}C\delta^{\frac{1}{p}}e\|c_n\|_{l^2_n}$, then we set
$$
r = [\frac{\lambda}{Ce\delta^{\frac{1}{p}}\|c_n\|_{l^2_n}}]^2 \quad (\geq p).
$$

So that (\ref{LpLarge}) yields (\ref{ExpLarge}).
\end{proof}

\begin{cor}\label{CorLp}
Let $p, q \geq 2$, and $P_N R  = \sum_{|n|\sim N}  \frac{g_n(\omega)}{\langle n\rangle^{d-1-\alpha}} e^{i(n\cdot x +|n|^2t|)}$, where $N$ is a dyadic coordinate. There exists $A\subset \Omega$, $C$ and $c>0$, with $\PPP(A) < Ce^{-\frac{1}{\delta^c}}$, such that for each $\omega \in A^c$ and each dyadic coordinate $N$, we have
$$
\|P_N R\|_{{L^p_tL^q_x([0, \delta] \times\mathbb{T}^d)}} \leq \delta^c \frac{\log N}{N^{s_c-\alpha}}.
$$
for $\delta < \delta_0$.
\end{cor}
\begin{proof}
By Lemma \ref{Lem2}, for each dyadic coordinate $N$, set $\lambda = \delta^{\frac{1}{2p}} {\log N}{\|P_N R\|_{L^2_x}}$, there exists $A_N\subset \Omega$, such that for $\omega \in A_N^c$, with $\PPP(A_N) < C\exp(-\frac{\log N}{\delta^{\frac{1}{p}}})$ we obtain that
$$
\|P_N R\|_{{L^p_tL^q_x([0, \delta]\times\mathbb{T}^d)}} \leq \delta^{\frac{1}{2p}} \frac{\log N}{N^{s_c-\alpha}},
$$
since $\|P_N R\|_{L^2_x} \sim \frac{1}{N^{s_c-\alpha}}$.

Set $A=\cup_{N} A_N$, and $c=\frac{1}{p}$, then we have
\begin{align*}
\PPP(A) &\leq \sum_N \PPP(A_N)
\leq \sum_N C \exp(-\frac{(\log N)^2}{\delta^{c}})\\
&\leq \sum_{k=1}^\infty Ce^{-\frac{k^2}{\delta^{c}}}
\leq \sum_{k=1}^\infty Ce^{-\frac{k}{\delta^{\frac{1}{p}}}}\\
&= \frac{Ce^{-\frac{1}{\delta^c}}}{1-e^{-\frac{1}{\delta^c}}}< 2C e^{-\frac{1}{\delta^c}}.
\end{align*}
when $\delta$ is small enough. 
\end{proof}

\begin{lem}[Proposition 3.1 in \cite{nahmod2015almost}]\label{MLD}
For fixed $n\in \ZZZ^d$, let
$$
D(n) = \{(n_1, n_2, n_3)\in \ZZZ^d\times
\ZZZ^d\times\ZZZ^d : n= n_1-n_2+n_3,\  n_2\neq n_1,\ n_2 \neq n_3, \ n_1\neq n_3\}.
$$
Given $\{c_{n_1, n_2, n_3}\}_{l_2(D(n))}$, define $F_n$ by
$$
F_n := \sum_{D(n)} c_{n_1, n_2, n_3} g_{n_1}(\omega) \overline{(g_{n_2})}(\omega) g_{n_3}(\omega).
$$
Then there exists $C>0$ such that for every $\lambda>0$ we have
$$
\PPP(\{\w : |F_n(\w)|>\lambda\})\leq \exp\left(\frac{-C\lambda^{2/3}}{\|F_n(\w)\|^{2/3}_{L^2(\Omega)}}\right).
$$
\end{lem}


\section{Function spaces}
In this section, we introduce $X^s$ and $Y^s$ spaces which are based on the atomic space $U^p$ and $V^p$ which were firstly applied to PDEs in \cite{hadac2009well}\cite{herr2011global}\cite{herr2014strichartz}.
$\mathcal{H}$ is a separable Hilbert space on $\mathbb{C}$, and $\mathcal{Z}$ denotes the set of finite partitions $-\infty = t_0 < t_1 < ... <t_K = \infty$ of the real line, with the convention that  $v(\infty) := 0$ for any function $v : \mathbb{R} \to \mathcal{H}$.
\begin{definition}[Definition 2.1 in \cite{herr2011global}]
Let $1\leq p < \infty$. For $\{t_k\}_{k=0}^K \in \mathcal{Z}$ and $\{\phi_k\}_{k=0}^{K-1} \subset \mathcal{H}$ with $\sum_{k=0}^K \|\phi_k\|_{\mathcal{H}}^p = 1$ and $\phi_0 = 0$. A $U^p$-atom is a piecewise defined function  $a : \mathbb{R} \to \mathcal{H}$ of the form
$$
a = \sum_{k=1}^K \mathds{1}_{[t_{k-1}, t_k)} \phi_{k-1}.
$$
The atomic Banach space $U^p(\mathbb{R}, \mathcal{H})$ is then defined to be the set of all functions $u : \mathbb{R} \to \mathcal{H}$ such that 
$$
u = \sum_{j=1}^{\infty} \lambda_{j} a_j,\qquad  \text{for}\   U^p \text{-atoms}\   a_j,\qquad \{\lambda_j\}_j \in \ell^1,
$$
with the norm
$$
\|u\|_{U^p} : = \inf \{\sum_{j=1}^{\infty} |\lambda_j| : u = \sum_{j=1}^\infty \lambda_j a_j,\  \lambda_j \in \mathbb{C}\ \text{and}\ a_j\ \text{an}\ U^p\ \text{atom}\}.
$$
Here $\mathds{1}_{I}$ denotes the indicator function over the time interval $I$.
\end{definition}

\begin{definition}[Definition 2.2 in \cite{herr2011global}]
Let $1\leq p < \infty$. The Banach space $V^p(\mathbb{R}, \mathcal{H})$ is defined to be the set of all functions $v: \mathbb{R} \to \mathcal{H}$ with $v(\infty):=0$ and $v(-\infty) := \lim_{t\to -\infty} v(t)$ exists, such that 
$$
\|v\|_{V^p} : = \sup_{\{t_k\}_{k=0}^K\in \mathcal{Z}} (\sum_{k=1}^K \|v(t_k)-v(t_{k-1})\|^p_\mathcal{H})^\frac{1}{p}\quad
 \text{is}\  \text{finite}.
$$
Likewise, let $V_{-}^p$ denote the closed subspace of all $v \in V^p$ with $\lim_{t\to -\infty}v(t) = 0$. $V_{-, rc}^p$ means all right-continuous $V_{-}^p$ functions.
\end{definition}

\begin{remark}[Some embeding properties]\label{rmk:embedding}
Note that for $1\leq p \leq q < \infty$,
\begin{equation}
U^p(\mathbb{R}, \mathcal{H}) \hookrightarrow U^q(\mathbb{R}, \mathcal{H}) \hookrightarrow  L^{\infty}(\mathbb{R},\mathcal{H}),
\end{equation}
and functions in $U^p(\mathbb{R}, \mathcal{H})$ are right continuous, and $\lim_{t\to -\infty} u(t) = 0$ for each $u \in U^p(\mathbb{R}, \mathcal{H})$.  Also note that,
\begin{equation}
U^p(\mathbb{R}, \mathcal{H}) \hookrightarrow V^p_{-,rc} (\mathbb{R}, \mathcal{H}) \hookrightarrow U^q(\mathbb{R}, \mathcal{H}).
\end{equation}
\end{remark}

\begin{definition}[Definition 2.5 in \cite{herr2011global}]
For $s\in \mathbb{R}$, we let $U^p_{\Delta}H^s$, respectively $V^p_{\Delta}H^s$, be the space of all functions $u : \mathbb{R}\to H^s(\mathbb{T}^d)$ such that $t\mapsto e^{-it\Delta}u(t)$ is in $U^p(\mathbb{R}, H^s)$, respectively in  $V^p(\mathbb{R}, H^s)$ with norm
$$
\|u\|_{U^p(\mathbb{R}, H^s)} := \|e^{-it\Delta}u(t)\|_{U^p(\mathbb{R},H^s)}, 
\quad \|u\|_{V^p(\mathbb{R}, H^s)} := \|e^{-it\Delta}u(t)\|_{V^p(\mathbb{R},H^s)}.
$$
\end{definition}

\begin{definition}[Definition 2.6 in \cite{herr2011global}]\label{def:Xs}
For $s\in \mathbb{R}$, we define $X^s$ as the space of all functions $u : \mathbb{R}\to H^s(\mathbb{T}^d)$ such that for every $n\in \mathbb{Z}^d$, the map $t\mapsto e^{it|n|^2}\widehat{u(t)}(n)$ is in $U^2(\mathbb{R}, \mathcal{C})$, and with the norm
\begin{equation}
\|u\|_{X^s} : = (\sum_{n\in \mathbb{Z}^d}\langle n \rangle^{2s} \|e^{it|n|^2}\widehat{u(t)}(n)\|^2_{U_t^2})^{\frac{1}{2}}\quad \text{is finite}. 
\end{equation}
\end{definition}

\begin{definition}[Definition 2.7 in \cite{herr2011global}]
For $s\in \mathbb{R}$, we define $Y^s$  as the space of all functions $u : \mathbb{R}\to H^s(\mathbb{T}^d)$ such that for every $n\in \mathbb{Z}^d$, the map $t\mapsto e^{it|n|^2}\widehat{u(t)}(n)$ is in $V_{rc}^2(\mathbb{R}, \mathcal{C})$, and with the norm
\begin{equation}
\|u\|_{Y^s} : = (\sum_{n\in \mathbb{Z}^d}\langle n \rangle^{2s} \|e^{it|n|^2}\widehat{u(t)}(n)\|^2_{V_t^2})^{\frac{1}{2}}\quad \text{is finite}. 
\end{equation}
\end{definition}

Note that 
\begin{equation}\label{eq:Embedding}
U^2_{\Delta}H^s \hookrightarrow X^s \hookrightarrow Y^s \hookrightarrow V^2_{\Delta}H^s.
\end{equation}

\begin{prop}[Proposition 2.10 in \cite{hadac2009well}]\label{EstimateFreeSolution}
Suppose $u :=  e^{it\Delta}\phi$ which is a free Schr\"{o}dinger solution, then we obtain that
$$\|u\|_{X^s([0,\dd])} \leq \|\phi\|_{H^s}.$$
\end{prop}
\begin{proof}
Since $u :=  e^{it\Delta}\phi$, then $\|u\|_{X^s}  = (\sum_{n\in \mathbb{Z}^d}\langle n \rangle^{2s} \|\widehat{\phi}(n)\|^2_{U_t^2})^{\frac{1}{2}} \leq \|\phi\|_{H^s}$.
\end{proof}

\begin{definition}[The corresponding restriction spaces to a time interval $I$]
\label{def:LocalNorm}
For $p\geq 1$ and a bounded time interval $I$. Define $U^p(I)$, $V^p(I)$, $X^s(I)$ and $Y^s(I)$ with the restriction norms:
\[
\|u\|_{U^p(I)} = \inf \{ \|\widetilde{u}\|_{U^p}:  \widetilde{u}(t) = u(t), t\in I\} \text{ and } \|u\|_{V^p(I)} = \inf \{ \|\widetilde{u}\|_{V^p}:  \widetilde{u}(t) = u(t), t\in I\};
\]
\[
\|u\|_{X^s(I)} = \inf \{ \|\widetilde{u}\|_{X^s}:  \widetilde{u}(t) = u(t), t\in I\} \text{ and } \|u\|_{Y^s(I)} = \inf \{ \|\widetilde{u}\|_{Y^s}:  \widetilde{u}(t) = u(t), t\in I\}.
\]

\end{definition}

\begin{prop}[Proposition 2.19 in \cite{hadac2009well}]\label{TransferPrinciple}
Let $T_0 : L_x^2 \times \cdots \times L_x^2 \to L^{1}_{x, loc}(\TTT^d)$ be an m-linear operator. Assume that for some $1\leq p \leq \infty$
\begin{equation}
\| T_0 (e^{it\Delta }\phi_1,\cdots, e^{it\Delta }\phi_m)\|_{L^p(\mathbb{R}\times\mathbb{T}^d)} \lesssim \prod_{i=1}^m\|\phi_i\|_{L^2(\mathbb{T}^d)}.
\end{equation}
Then, there exists an extension $T : U^p_{\Delta}\times \cdots \times U^p_{\Delta} \to L^p(\mathbb{R}\times\mathbb{T}^d)$ satisfying
\begin{equation}
\|T(u_1, \cdots, u_m)\|_{L^p(\mathbb{R}\times\mathbb{T}^d)} \lesssim \prod_{i=1}^m \|u_i\|_{U^p_{\Delta}};
\end{equation}
and  such that $T(u_1, \cdots , u_m) (t, \cdot) = T_0 (u_1(t), \cdots , u_m(t))(\cdot)$, a.e.
\end{prop}

\begin{prop}\label{timelocalTransferPrinciple}
Let $T_0 : L_x^2 \times \cdots \times L_x^2 \to L^{1}_{x, loc}(\TTT^d)$ be m-linear operator. Assume that for some bounded time interval $I\subset \RRR$, and $1< q \leq \infty$ 
\begin{equation}\label{eq:timelocalTransferPrinciple2}
\left|\int_{J}\int_{\TTT^d} T_0 (e^{it\Delta }\phi_1,\cdots, e^{it\Delta }\phi_m)\, dxdt\right| \lesssim |J|^{\frac{1}{q}}\prod_{i=1}^m\|\phi_i\|_{L^2(\mathbb{T}^d)}, \qquad \text{for any } J\subset I.
\end{equation}

Then, for $1\leq p<\infty$ satisfying $\frac{1}{p}+\frac{1}{q} = 1$, there exists an extension $T : U^p_{\Delta}\times \cdots \times U^p_{\Delta} \to L^1_{x,t, loc}(I\times \TTT^d)$ satisfying
\begin{equation}\label{eq:timelocalTransferPrinciple1}
\left| \int_I \int_{\TTT^d} T(u_1, \cdots, u_m)\, dxdt \right| \lesssim |I|^{\frac{1}{q}}\prod_{i=1}^m \|u_i\|_{U^p_{\Delta}(I)};
\end{equation}
and  such that $T(u_1, \cdots , u_m) (t, \cdot) = T_0 (u_1(t), \cdots , u_m(t))(\cdot)$, a.e.
\end{prop}

\begin{remark}
In Hadac-Herr-Koch's paper \cite{hadac2009well}, they derived a "transfer principle" as Proposition \ref{TransferPrinciple}, which consider the $L^p$ norm of the multilinear operator $T$ over the whole time space $\RRR$, while Proposition \ref{timelocalTransferPrinciple} focus on the integral in time (or actually $L^1$ norm is also fine) on a finite time interval $I$. By a stronger assumption (which gives some better estimates on each small intervals J), Proposition \ref{timelocalTransferPrinciple} somehow takes advantage of the finite time interval to improve the bounds from $U^1$ norm to $U^p$. In the following proof of Proposition \ref{NonlinearEstimate}, the \textbf{Case B} heavily relies on Proposition \ref{timelocalTransferPrinciple}.
\end{remark}

\begin{proof}
By multi-linearity of $T_0$ and definition of $U^p$ norm, it will suffice to show that (\ref{eq:timelocalTransferPrinciple1}) is true for all $U_\Delta^p$-atoms $u_i$. Let $a_1, \cdots, a_m$ be $U_\Delta^p$-atoms given as 
\[
a_i = \sum_{k_i = 1}^{K_i} \mathds{1}_{\scriptsize{I_{k_i, i}}}e^{it\Delta}\phi_{k_i-1,i},\quad \text{for } i = 1, \cdots, m.
\]
where $I_{k_i, i} = [t_{k_i-1, i},t_{k_i, i})$, and
such that
 \begin{equation}\label{eq:atom}
\sum_{k_i =1}^{K_i} \|\phi_{k_i-1,i}\|_{L^2_x}^p = 1. 
\end{equation}
 Then, by (\ref{eq:timelocalTransferPrinciple2}), Cauchy-Schwartz inequality and by induction,

\begin{align}\notag
&\left|\int_I \int_{\TTT^d} T(a_1, \cdots, a_m)(t)\,dxdt\right|\\\notag
\leq&  \sum_{\scriptsize{\begin{subarray}{c} 1\leq k_1 \leq K_1 \\ \cdots \\ 1\leq k_m\leq K_m\end{subarray}}} \left| \int_{\cap_{i=1}^m I_{k_i, i}} \int_{\TTT^d} T_0 (e^{it\Delta}\phi_{k_1-1,1}, \cdots, e^{it\Delta}\phi_{k_m-1,m})\,dxdt \right|\\\label{eq:atom1}
\leq & \sum_{\scriptsize{\begin{subarray}{c} 1\leq k_1 \leq K_1 \\ \cdots \\ 1\leq k_m\leq K_m\end{subarray}}} |\cap_{i=1}^m I_{k_i, i}|^{\frac{1}{q}} \prod_{i=1}^m \|\phi_{k_i-1,i}\|_{L^2_x}\\\label{eq:atom2}
\leq & \sum_{\scriptsize{\begin{subarray}{c} 1\leq k_2 \leq K_2 \\ \cdots \\ 1\leq k_m\leq K_m\end{subarray}}}  \prod_{i=2}^m \|\phi_{k_i-1,i}\|_{L^2_x}
\left(\sum_{1\leq k_1\leq K_1} |\cap_{i=1}^m I_{k_i, i}|  \right)^{\frac{1}{q}} 
\left(\sum_{1\leq k_1\leq K_1} \|\phi_{k_1-1,1}\|_{L_x^2}^p  \right)^{\frac{1}{p}}
\end{align}

For fixed $k_2, k_3, \cdots, k_m$, since
\[
I_{k_2,2}\cap\cdots\cap I_{k_m, m} = \cup_{1\leq k_1\leq K_1} \left(I_{k_1, 1}\cap I_{k_2,2}\cap \cdots\cap I_{k_m,m}\right),
\]
we have 
\begin{equation}\label{eq:Ik}
\left(\sum_{1\leq k_1\leq K_1} |\cap_{i=1}^m I_{k_i, i}|  \right)^{\frac{1}{q}} =  |\cap_{i=2}^m I_{k_i, i}|^{\frac{1}{q}}.
\end{equation}

Based on (\ref{eq:atom}) (\ref{eq:Ik}) (\ref{eq:atom2}), we obtain that
\begin{align*}
&|\int_I \int_{\TTT^d} T(a_1, \cdots, a_m)(t)\,dxdt|\\
\leq & \sum_{\scriptsize{\begin{subarray}{c} 1\leq k_2 \leq K_2 \\ \cdots \\ 1\leq k_m\leq K_m\end{subarray}}}  \prod_{i=2}^m \|\phi_{k_i-1,i}\|_{L^2_x}
\left(\sum_{1\leq k_1\leq K_1} |\cap_{i=1}^m I_{k_i, i}|  \right)^{\frac{1}{q}} 
\left(\sum_{1\leq k_1\leq K_1} \|\phi_{k_1-1,1}\|_{L_x^2}^p  \right)^{\frac{1}{p}}\\
\leq & \sum_{\scriptsize{\begin{subarray}{c} 1\leq k_2 \leq K_2 \\ \cdots \\ 1\leq k_m\leq K_m\end{subarray}}} |\cap_{i=2}^m I_{k_i, i}|^{\frac{1}{q}} \prod_{i=2}^m \|\phi_{k_i-1,i}\|_{L^2_x}
\end{align*}

If we iterate (\ref{eq:atom1}) (\ref{eq:atom2}) on $k_2, k_3, \cdots, k_m$, finally we obtain that

\begin{align*}
&|\int_I \int_{\TTT^d} T(a_1, \cdots, a_m)(t)\,dxdt|\\
\leq & \sum_{\scriptsize{\begin{subarray}{c} 1\leq k_2 \leq K_2 \\ \cdots \\ 1\leq k_m\leq K_m\end{subarray}}} |\cap_{i=2}^m I_{k_i, i}|^{\frac{1}{q}} \prod_{i=2}^m \|\phi_{k_i-1,i}\|_{L^2_x}\\
\leq & \sum_{\scriptsize{\begin{subarray}{c} 1\leq k_3 \leq K_3 \\ \cdots \\ 1\leq k_m\leq K_m\end{subarray}}} |\cap_{i=3}^m I_{k_i, i}|^{\frac{1}{q}} \prod_{i=3}^m \|\phi_{k_i-1,i}\|_{L^2_x}\\
&\cdots\\
\leq & |I|^{\frac{1}{q}}.
\end{align*}
So we obtain (\ref{eq:timelocalTransferPrinciple1}).
\end{proof}

\begin{prop}[Proposition 2.20 in \cite{hadac2009well}]\label{Interpolation1}
Let $q_1, ... , q_m > 2$ ($m \in \mathbb{N}$), $E$ be a Banach space and $T: U_{\Delta}^{q_1} \times \cdots \times U_{\Delta}^{q_m} \to E$ be a bounded m-linear operator with
\begin{equation}
\|T(u_1, \cdots , u_m)\|_{E} \leq C \prod_{i=1}^m \|u_i\|_{U_{\Delta}^{q_i}}.
\end{equation} 
And also assume there exists $0<C_2<C$ such that we hold,
\begin{equation}
\|T(u_1, \cdots , u_m)\|_{E} \leq C_2 \prod_{i=1}^m \|u_i\|_{U_{\Delta}^{2}}.
\end{equation}
Then, $T$ satisfies the estimate
\begin{equation}
\|T(u_1, \cdots , u_m)\|_{E} \leq C_2(\log\frac{C}{C_2}+1) \prod_{i=1}^m \|u_i\|_{V_{\Delta}^{2}}, \quad u_i\in V^2_{rc}, \ i=1,...,m.
\end{equation}
\end{prop}

To make the proposition \ref{Interpolation1} suitable for the following nonlinear estimates, we also need to introduce a similar interpolation proposition for the integral of $T$ over a time interval $I$ as following:

\begin{prop}\label{Interpolation2}
Let $q_1, ... , q_m > 2$ ($m \in \mathbb{N}$),  and $T: U_{\Delta}^{q_1} \times \cdots \times U_{\Delta}^{q_m} \to L^1_{x,t, loc}(I\times \TTT^d)$ be a m-linear operator with
\begin{equation}
\left|\int_{I}\int_{\TTT^d} T (u_1,\cdots, u_m)\, dxdt\right| \leq C \prod_{i=1}^m \|u_i\|_{U_{\Delta}^{q_i}}.
\end{equation} 
And also assume there exists $0<C_2<C$ such that we hold,
\begin{equation}
\left|\int_{I}\int_{\TTT^d} T (u_1,\cdots, u_m)\, dxdt\right| \leq C_2 \prod_{i=1}^m \|u_i\|_{U_{\Delta}^{2}}.
\end{equation}
Then, $T$ satisfies the estimate
\begin{equation}
\left|\int_{I}\int_{\TTT^d} T (u_1,\cdots, u_m)\, dxdt\right|\leq C_2(\log\frac{C}{C_2}+1) \prod_{i=1}^m \|u_i\|_{V_{\Delta}^{2}}, \quad u_i\in V^2_{rc,\Delta}, \ i=1,...,m.
\end{equation}
\end{prop}
\begin{proof}
The proof is almost the same as that of Proposition 2.20 in \cite{hadac2009well}, since $\left|\int_{I}\int_{\TTT^d} T (u_1, \cdots, u_m)\, dxdt\right|$ is m-sublinear for $u_1, \cdots, u_m$ as $\|T(u_1, \cdots , u_m)\|_{E}$ in Prop \ref{Interpolation1}.
\end{proof}

\begin{definition}[Duhamel operator]\label{DuhamelOperator}
Let $f \in L_{loc}^1 ([0,\infty), L^2(\mathbb{T}^4))$, and we define the Duhamel operator $\mathcal{I}$
\begin{equation}
\mathcal{I}(f)(t): = \int_0^t e^{i(t-t')\Delta} f(t') dt',
\end{equation}
for $t>0$  and $\mathcal{I}(f)(t) := 0$ otherwise.
\end{definition}

\begin{prop}[Proposition 2.11 in \cite{herr2011global}]\label{DualProp}
Let $s>0$, and a time interval $I = [0, \delta]$. For $f\in L^1(I, H^s(\mathbb{T}^4))$ we have $\mathcal{I}(f) \in X^s(I)$ and 
\begin{equation}
\| \mathcal{I}(f) \|_{X^s(I)}\leq \sup_{\|v\|_{Y^{-s}(I)}=1}\left|\int_0^\delta
\int_{\mathbb{T}^d} f(t, x)\overline{v(t,x)}dxdt\right|.
\end{equation}
\end{prop}


\section{Auxiliary lemmata and notations}

\begin{definition}[Littlewood-Paley decomposition]
For $N>1$ a dyadic number, we denote by $P_{\leq N}$ the rectangular Fourier projection operator:
$$
P_{\leq N}f = \sum_{n\in\mathbb{Z}^4: |n_i|\leq N} \hat{f}(n)e^{in\cdot x}.
$$
Then $P_N = P_{\leq N}-P_{\leq N-1}$. Moreover, if $C$ is a subset of $\mathbb{Z}^d$, then the Fourier projection operator onto $C$ is defined by $P_C$
$$
P_{C}f = \sum_{n\in\mathbb{Z}^d: n\in C} \hat{f}(n)e^{in\cdot x}.
$$
\end{definition}

In the Bourgain's GAFA paper\cite{bourgain1993fourier}, he firstly introduced the following Strichartz estimate of Schr\"odinger operator on tori as a conjucture, and proved parts of the conjucture. And then Bourgain-Demeter\cite{bourgain2015proof} proved the following Strichartz estimate. 

\begin{prop}[Strichartz estimate\cite{bourgain1993fourier}\cite{bourgain2015proof}]\label{Strichartz}
Let $p>p_c$, where $p_c = \frac{2(d+2)}{d}$. For all $N\geq 1$ we have
\begin{eqnarray}
\|P_{N} e^{it\Delta} \phi \|_{L^p_{x,t}(\mathbb{T}\times\mathbb{T}^d)}&\lesssim& N^{\frac{d}{2}-\frac{d+2}{p}}\|P_{N}\phi\|_{L^2_x(\mathbb{T}^d)},\label{eq:Strichartz0}\\
\|P_{C} e^{it\Delta} \phi \|_{L^p_{x,t}(\mathbb{T}\times\mathbb{T}^d)}&\lesssim& N^{\frac{d}{2}-\frac{d+2}{p}}\|P_{C}\phi\|_{L^2_x(\mathbb{T}^d)},\label{eq:Strichartz1}\\
\|P_{N} u \|_{L^p_{x,t}(\mathbb{T}\times\mathbb{T}^d)}&\lesssim& N^{\frac{d}{2}-\frac{d+2}{p}}\|P_{C} u \|_{U_\Delta^{p}L^2}\label{eq:Strichartz3},\\
\|P_{C} u \|_{L^p_{x,t}(\mathbb{T}\times\mathbb{T}^d)}&\lesssim& N^{\frac{d}{2}-\frac{d+2}{p}}\|P_{C} u \|_{U_\Delta^{p}L^2}\label{eq:Strichartz2},
\end{eqnarray}
where $C$ is a cube in $\mathbb{Z}^d$ with sides parallel to the axis of side length $N$.
\end{prop}

Note that the last inequality (\ref{eq:Strichartz3})(\ref{eq:Strichartz2}) follows (\ref{eq:Strichartz0})(\ref{eq:Strichartz1}) and Proposition \ref{TransferPrinciple}.

\begin{lem}[Integer lattice counting estimates \cite{bourgain1993fourier}]\label{lem:BourgainCounting}
Denote the number of set $\{(X_1,\cdots, X_d)\in \mathbb{Z}^d: X_1^2+\cdots+X_d^2 = A\}$ by $C_{d,A}$. Then $C_{d,A}$ can be bounded by
$$
\begin{aligned}
&A^\epsilon &(d=2)\\
&A^{\frac{1}{2}+\epsilon}&(d=3)\\
&A^{1+\epsilon}&(d=4)\\
&A^{\frac{d-2}{2}}&(d>4)
\end{aligned}.
$$
where $\epsilon$ is an arbitrary small positive number.
\end{lem}

By Lemma \ref{lem:BourgainCounting}, it's easy to obtain the following lattice counting lemmas.
\begin{lem}\label{Counting1}
Let $S_R$ be a sphere of radius $R$, $B_r$ be a ball of radius $r$, and $\mathcal{P}$ be a plane in $\RRR^d$ for $d\geq 3$. Then
\begin{align}
&|\ZZZ^d \cap S_R|\leq R^{d-2+\epsilon},\\
&|\ZZZ^d \cap B_r \cap \mathcal{P}|\leq r^{d-1},
\end{align}
where $|\cdot|$ denotes cardinality and $\epsilon$ is an arbitrary small positive number.
\end{lem}

\begin{lem}\label{Counting2}
Consider the set
\[
S = \{ (n_1,n_2,n_3)\in \ZZZ^d\times\ZZZ^d\times\ZZZ^d : n_2\neq n_1, n_3, |n_i|\sim N_i \text{ for } i =1,2,3, \text{ and } \langle n_2-n_1, n_2-n_3 \rangle = \mu\}.
\]
For a fixed $n_2$, $|S(n_2)| \lesssim N_1^{d-1} N_3^{d-1} \min \{N_1, N_3\}^\epsilon$, where $|\cdot|$ denotes cardinality and $\epsilon$ is an arbitrary small positive number.
\end{lem}

\begin{lem}[Bounds of Fourier coefficients of Characteristic function]\label{lem:boundFourier}
Consider $\mathds{1}_{[a,b]}(t)$ as a function in $L^2([0, 2\pi])$ where $a, b \in [0, 2\pi]$, then the Fourier coefficients $|\mathcal{F}{(\mathds{1}_{[a, b]})} (k)| \leq \frac{2}{|k|}$ for all $k\in \ZZZ$.
\end{lem}
\begin{proof}
$|\mathcal{F}{(\mathds{1}_{[a, b]})} (k)| = |\frac{e^{ikb}- e^{ika}}{ik}| \leq \frac{2}{|k|}$.
\end{proof}

\begin{lem}[Lemma 6.3 in \cite{nahmod2015almost}]\label{lem:matrix}
Let $\mathcal{A} = (A_{ik})_{1\leq i \leq N\atop
1\leq k\leq M}$ be an $N\times M$ matrix. Then
\[
\|\mathcal{A}\mathcal{A}^*\|\leq \max_{1\leq j\leq N}\sum_{k=1}^M|A_{jk}|^2 +
\left(\sum_{i\neq j} \left|\sum_{k=1}^M A_{ik}\overline{A_{jk}}\right|^2\right)^{\frac{1}{2}}
\]
where $\|\cdot \|$ is the $2$-norm.
\end{lem}


\section{Estimate for nonlinear term}
To estimate 
$\| \mathcal{I}(\mathcal{N}(w + v_0^\omega)) \|_{X^s([0, \delta])}$, by Prop \ref{DualProp} the we just need to bound the integral $\int_0^\delta\int_{\mathbb{T}^4} \mathcal{N}(w + v_0^\omega)\overline{u^{(0)}}dxdt$, where $\delta<1$. This section will focus on estimating this integral.

\begin{prop}\label{NonlinearEstimate}
Suppose $d\geq 3$ and   $s_r(d)$ is given in (\ref{coef:srd}).
Let $0\leq \alpha< s_r(d)$, $s\in [s_c, s_c + s_r(d)-\alpha)$, $r>0$, $0\leq \dd<1$, and $I = [0,\dd]$.    There exist $\Omega_\dd\subset \Omega$ with $\PPP(\Omega_\dd^c)<e^{-1/\dd^r}$, and $c>0$, such that we obtain that
\begin{align*}
&\left|\int_0^\delta\int_{\mathbb{T}^d} \mathcal{N}(w^{(1)} + v_0^\omega, \overline{w^{(2)} + v_0^\omega}, w^{(3)} + v_0^\omega)\overline{u^{(0)}}\,dxdt\right|\\
\lesssim & \|u^{(0)}\|_{Y^{-s}(I)}\left(\dd^{c\min{\{1, s-s_c\}}}\|w^{(1)}\|_{X^s(I)}\|w^{(2)}\|_{X^s(I)}\|w^{(3)}\|_{X^s(I)}+\dd^{c}\sum_{S_J\subset\{1, 2, 3\}\atop J\neq \{1,2,3\}}\prod_{j\in S_J}\|w^{(j)}\|_{X^s(I)}\right),
\end{align*}
where $v_0^\omega$ is defined (\ref{def:v0}), $u^{(0)}\in Y^{-s}(I)$
 and $w^{(i)}\in X^s(I)$ for $i=1, 2, 3$. 
 (when the subset $S_J = \emptyset$, $\prod_{j\in S_J}\|w^{(j)}\|_{X^s(I)} =1.$)
\end{prop}

To show Proposition \ref{NonlinearEstimate}, it is clear that $\mathcal{N}(w + v_0^\omega)$ can be expressed as \begin{equation}
\sum_{u^{(i)}\in\{w,v_0^\omega\}, i=1,2,3}\mathcal{N}(u^{(1)}, \overline{u^{(2)}}, u^{(3)}).
\end{equation}

We dyadic decompose
 \[u_i = P_{N_i} u^{(i)},\text{ where } i\in \{0, 1, 2, 3\}.\] 

By the symmetry, in the following paper we suppose that $N_1\geq N_2 \geq N_3$, and we need to estimate the following integral case by case,
\begin{equation}\label{eq:NonlinearIntegral}
\int_0^\dd\int_{\mathbb{T}^d} \mathcal{N}(\widetilde{u}_{1}, \widetilde u_2, \widetilde u_3) \overline u_0\, dxdt,
\end{equation}
{where } $\widetilde{u_i} = u_i \text{ or } \overline{u_i}$ and only one of $\widetilde{u_i}$ can be $\overline{u_i}$.
\begin{remark}
To make the integral $\int_0^\dd\int_{\mathbb{T}^d} \widetilde{u}_{1} \widetilde u_2 \widetilde u_3 \overline u_0\, dxdt$ (which is the main term of (\ref{eq:NonlinearIntegral})) nontrivial, the two highest frequencies must be comparable, which means $N_1 \sim \max \{N_0, N_2\}$
($\frac{1}{4} N_1 \leq  \max {\{N_0, N_2\}} \leq 4 N_1$). It is easy to show if $\frac{1}{4} N_1 >  \max {\{N_0, N_2\}}$ or $\max {\{N_0, N_2\}} > 4N_1$, then the integral $\int_0^\dd\int_{\mathbb{T}^d} \widetilde{u}_{1} \widetilde u_2 \widetilde u_3 \overline u_0\, dxdt$ is zero. Then the following two cases need to considered: 
\begin{itemize}
\item $N_0\sim N_1\geq N_2$;
\item $N_0< N_2 \sim N_1$.
\end{itemize}
\end{remark}

Now let's summarize all cases of $(u_1, u_2, u_3)$ we should consider. 
Denote \begin{equation}\label{def:DR}
R_i = P_{N_i} v_0^\omega \text{ and }  D_i = P_{N_i} w \text{ for } i\in\{1, 2, 3\}.
\end{equation} 
The list of all cases of $(u_1, u_2, u_3)$ is below:

\begin{enumerate}
\item[A.]  $u_1 = D_1$: 
\begin{enumerate}
\item $(D_1, D_2, D_3)$;
\item $(D_1, D_2, R_3)$;
\item $(D_1, R_2, D_3)$;
\item $(D_1, R_2, R_3)$;
\end{enumerate}
\item[B.] $u_1 = R_1$:
\begin{enumerate}
\item $(R_1, R_2, R_3)$;
\item $(R_1, R_2, D_3)$;
\item $(R_1, D_2, R_3)$;
\item $(R_1, D_2, D_3)$.
\end{enumerate}
\end{enumerate}

\subsection{\textbf{Case A (a)}}
We consider the all deterministic case $u_i= D_i$ for all $i\in\{1,2,3\}$. It's directly the local well-posed result for the critical data following the strichartz estimates Proposition \ref{Strichartz} (the case $d=4$ is in \cite{KV}). 
\begin{prop}\label{AllDet}
Assume $N_i$, $i=0,1,2,3$, are dyadic numbers and $N_1\geq N_2 \geq N_3$, and $0 
\leq\dd \leq 1$ . For $s\geq s_c$, there exists  $c>0$, so that we can bound the integral:
$$\left|\int_0^\dd\int_{\mathbb{T}^d} \mathcal{N}(\widetilde{D}_{1}, \widetilde D_2, \widetilde D_3) \overline u_0 \, dxdt\right| \lesssim \dd^{c \min\{1, s-s_c\}} (\frac{N_3\min{\{N_0, N_2\}}}{N_2^2})^c \frac{1}{N_2^{s-s_c}} \|{u}_0\|_{Y^{-s}}\|{D}_1\|_{X^s}\|{D}_2\|_{X^{s}}\|{D}_3\|_{X^{s}}$$
where $u_0$, $\widetilde{D}_{1}$, $\widetilde D_2$, and $\widetilde D_3$ is defined as (\ref{def:DR}).
\end{prop}
\begin{proof}

We decompose $\RRR^d = \cup_j C_j$, where each $C_j$ is a cube of side-length $N_2$. Let $P_{C_j}$ denote the family of Fourier
projections onto the cube $C_j$. We write $C_j \sim C_k$ if the sum set $\{c_1+c_2: c_1\in C_j, c_2\in C_k\}$ overlaps the Fourier support of $P_{\leq 2N_2}$. Observe that given $C_k$ there are a bounded
number of $C_j \sim C_k$. If $N_0\sim N_1\geq N_2$, and we decompose $u_0$ and $D_1$ with Fourier projections onto the small cubes of size $N_2$. If $N_0< N_2 \sim N_1$, and we also decompose $u_0$ and $D_1$ with Fourier projections onto the cubes of size $N_2$, however the frequency of $u_0$ has Fourier support of $P_{\leq N_0}$ which is only in one cube of size $N_2$. For the case of $N_0< N_2 \sim N_1$, the cube decomposition doesn't help, but for simplicity of notations, we use the same cube decomposition.

\noindent\textit{(1) Case: $d=3$}

First, let's consider $\N_1 (\widetilde{D}_1, \widetilde{D}_2, \widetilde{D}_3 ) = \pm \widetilde{D}_1\widetilde{D}_2\widetilde{D}_3$.
Set $\frac{11}{2}^+$ satisfying
\begin{equation}\label{coef:pq3d}
\frac{1}{\frac{11}{2}^+} = \frac{2}{11}-c_1,
\end{equation} 
where $c_1 = \min{\{\frac{2}{11}-\epsilon, \frac{2}{5}(s-s_c)\}}$. (In this paper, we always use $\epsilon$ as a small positive number which can be chosen arbitrarily small, and $\epsilon$ may be different in the different positions.)

By the cube decomposition, and H\"{o}lder inequality, we obtain that
\begin{align}\label{eq:Aa1}
& \left|\int_0^\dd\int_{\mathbb{T}^3}\overline{u}_0\widetilde{D}_1\widetilde{D}_2\widetilde{D}_3 \, dxdt\right|\\\notag
 \leq& \sum_{C_j\sim C_k}\left|\int_0^\dd\int_{\mathbb{T}^3}(P_{C_j}\overline{u}_0)(P_{C_k}\widetilde{D}_1)\widetilde{D}_2\widetilde{D}_3 \, dxdt\right|\\\notag
 \lesssim& \sum_{C_j\sim C_k} \|P_{C_j}u_0\|_{L_{t,x}^{\frac{11}{3}}}\|P_{C_k}D_1\|_{L_{t,x}^{\frac{11}{3}}}\|D_2\|_{L_{t,x}^{\frac{11}{3}}}\|D_3\|_{L_{t,x}^{\frac{11}{2}}}\\\label{eq:Aa2}
 \leq& \dd^{c_1}\sum_{C_j\sim C_k} \|P_{C_j}u_0\|_{L_{t,x}^{\frac{11}{3}}}\|P_{C_k}D_1\|_{L_{t,x}^{\frac{11}{3}}}\|D_2\|_{L_{t,x}^{\frac{11}{3}}}\|D_3\|_{L_{t,x}^{\frac{11}{2}^+}}.
 \end{align}
By Strichartz estimates (Lemma \ref{Strichartz}) and (\ref{eq:Aa2}), we obtain that
 \begin{align*}
 (\ref{eq:Aa1})\leq & \dd^{c_1}\sum_{C_j\sim C_k}\min{\{N_0, N_2\}}^{\frac{3}{22}}N_2^{\frac{6}{22}} N_3^{\frac{13}{22}+5c_1}\|P_{C_j}u_0\|_{Y^0}\|P_{C_k}D_1\|_{X^0}\|D_2\|_{X^0}\|D_3\|_{X^0}\\
 \lesssim & \dd^{c_1}(\frac{\min\{N_0, N_2\}}{N_2})^{\frac{3}{22}}(\frac{N_3}{N_2})^{\frac{1}{11}} N_3^{2(s-s_c)}\sum_{C_j\sim C_k} \|P_{C_j}u_0\|_{Y^{-s}}\|P_{C_k}D_1\|_{X^s}\|D_2\|_{X^{s_c}}\|D_3\|_{X^{s_c}}\\
 \lesssim &  \dd^{c\min\{1, (s-s_c)\}} (\frac{N_3\min{\{N_0, N_2\}}}{N_2^2})^c \frac{1}{N^{s-s_c}_2} \|{u}_0\|_{Y^{-s}}\|{D}_1\|_{X^s}\|{D}_2\|_{X^{s}}\|{D}_3\|_{X^{s}},
\end{align*}
where $c = \frac{1}{11}$.

Second, let's consider $\N_2(\widetilde{D}_1, \widetilde{D}_2, \widetilde{D}_3) = \pm \widetilde{D}_1 \int_{\mathbb{T}^3} \widetilde{D}_2 \widetilde{D}_3 dx$.
\begin{align*}
& \left|\int_0^\dd\int_{\mathbb{T}^3}\overline{u}_0\widetilde{D}_1dx \int_{\mathbb{T}^3}\widetilde{D}_2\widetilde{D}_3 dxdt\right|\\
\leq& \sum_{C_j\sim C_k}\left|\int_0^\dd\int_{\mathbb{T}^3}(P_{C_j}\overline{u}_0)(P_{C_k}\widetilde{D}_1)dx\int_{\TTT^3}\widetilde{D}_2\widetilde{D}_3 dxdt\right|\\
\lesssim & \sum_{C_j\sim C_k} \|P_{C_j}u_0\|_{L_{t}^{\frac{11}{3}}L_x^2}\|P_{C_k}D_1\|_{L_{t}^{\frac{11}{3}}L_x^2}\|D_2\|_{L_{t}^{\frac{11}{3}}L^2_x}\|D_3\|_{L_{t}^{\frac{11}{2}}L^2_x}\\
 \lesssim& \sum_{C_j\sim C_k} \|P_{C_j}u_0\|_{L_{t,x}^{\frac{11}{3}}}\|P_{C_k}D_1\|_{L_{t,x}^{\frac{11}{3}}}\|D_2\|_{L_{t,x}^{\frac{11}{3}}}\|D_3\|_{L_{t,x}^{\frac{11}{2}}}.
\end{align*}

Then we follow the same approach for $\mathcal{N}_1$ term,  we can hold the same bound of $\N_2$.

\noindent\textit{(2) Case: $d\geq 4$}

First, let's consider $\N_1 (\widetilde{D}_1, \widetilde{D}_2, \widetilde{D}_3 ) = \pm \widetilde{D}_1\widetilde{D}_2\widetilde{D}_3$.
Set $3^{+}$ and $\infty^-$ satisfying the following conditions:
\begin{equation}\label{coef:pq4d}
\frac{1}{\infty^-} = c_2,\qquad \frac{1}{3^+} = \frac{1}{3}-\frac{c_2}{3}.
\end{equation}
where $c_2 = \frac{2}{d+2}\min\{\frac{1}{4}, s-s_c\} +  c_3$, $c_3 = \frac{1}{d+2} \min \{s-s_c, \frac{1}{3}\}$.

By the cube decomposition, H\"{o}lder inequality, and Lemma \ref{Strichartz}, we obtain that
\begin{align*}
& \left|\int_0^\dd\int_{\mathbb{T}^d}\overline{u}_0\widetilde{D}_1\widetilde{D}_2\widetilde{D}_3 \, dxdt\right|\\
 \leq& \sum_{C_j\sim C_k}\left|\int_0^\dd\int_{\mathbb{T}^d}(P_{C_j}\overline{u}_0)(P_{C_k}\widetilde{D}_1)\widetilde{D}_2\widetilde{D}_3 \, dxdt\right|\\
 \lesssim& \sum_{C_j\sim C_k} \|P_{C_j}u_0\|_{L_{t,x}^{3^+}}\|P_{C_k}D_1\|_{L_{t,x}^{3^+}}\|D_2\|_{L_{t,x}^{3^+}}\|D_3\|_{L_{t,x}^{\infty^-}}\\
 \leq & \dd^{c_3}\sum_{C_j\sim C_k}\min\{N_0, N_2\}^{\frac{d}{6}-\frac{2}{3}+\frac{(d+2)c_2}{3}}N_2^{\frac{d}{3}-\frac{4}{3}+\frac{2(d+2)c_2}{3}} N_3^{\frac{d}{2}-(d+2)(c_2-c_3)}\\
 &\times\|P_{C_j}u_0\|_{Y^0}\|P_{C_k}D_1\|_{X^0}\|D_2\|_{X^0}\|D_3\|_{X^0}\\
 \lesssim & \dd^{c_3}(\frac{\min\{N_0, N_2\}}{N_2})^{\frac{d}{6}-\frac{2}{3}+\frac{(d+2)c_2}{3}}  (\frac{N_3}{N_2})^{\max{\{\frac{1}{2}, 1- 2(s-s_c)\}}-\min \{s-s_c, \frac{1}{3}\}} N_3^{\min \{s-s_c, \frac{1}{3}\}}\\
 &\times\sum_{C_j\sim C_k} \|P_{C_j}u_0\|_{Y^{-s}}\|P_{C_k}D_1\|_{X^s}\|D_2\|_{X^{s_c}}\|D_3\|_{X^{s_c}}\\
 \lesssim &  \dd^{c\min\{1, s-s_c\}} (\frac{N_3\min\{N_0, N_2\}}{N_2^2})^c\frac{1}{N_2^{s-s_c}} \|{u}_0\|_{Y^{-s}}\|{D}_1\|_{X^s}\|{D}_2\|_{X^{s}}\|{D}_3\|_{X^{s}},
\end{align*}
where $c = \frac{1}{3(d+2)}$ (it's easy to check that $\max{\{\frac{1}{2}, 1- 2(s-s_c)\}}-\min \{s-s_c, \frac{1}{3}\}\geq \frac{1}{6}> c$).

Second, let's consider $\N_2(\widetilde{D}_1, \widetilde{D}_2, \widetilde{D}_3) = \pm \widetilde{D}_1 \int_{\mathbb{T}^d} \widetilde{D}_2 \widetilde{D}_3 dx$.
\begin{align*}
& \left|\int_0^\dd\int_{\mathbb{T}^d}\overline{u}_0\widetilde{D}_1dx \int_{\mathbb{T}^d}\widetilde{D}_2\widetilde{D}_3 dxdt\right|\\
\leq& \sum_{C_j\sim C_k}\left|\int_0^\dd\int_{\mathbb{T}^d}(P_{C_j}\overline{u}_0)(P_{C_k}\widetilde{D}_1)dx\int_{\TTT^d}\widetilde{D}_2\widetilde{D}_3 dxdt\right|\\
\lesssim & \sum_{C_j\sim C_k} \|P_{C_j}u_0\|_{L_{t}^{3^+}L_x^2}\|P_{C_k}D_1\|_{L_{t}^{3^+}L_x^2}\|D_2\|_{L_{t}^{3^+}L^2_x}\|D_3\|_{L_{t}^{\infty^-}L^2_x}\\
 \lesssim& \sum_{C_j\sim C_k} \|P_{C_j}u_0\|_{L_{t,x}^{3^+}}\|P_{C_k}D_1\|_{L_{t,x}^{3^+}}\|D_2\|_{L_{t,x}^{3^+}}\|D_3\|_{L_{t,x}^{\infty^-}}.
\end{align*}

Then following the same approach for $\mathcal{N}_1$ term,  we can hold the bound of $\N_2$.

\end{proof}

\subsection*{\textbf{Case A (b)}}
We consider the case $(u_1, u_2, u_3)=(D_1, D_2, R_3)$.
\begin{prop}
Assume $N_i$, $i=0,1,2,3$, are dyadic numbers and $N_1\geq N_2 \geq N_3$, and $0 
\leq\dd \leq 1$ . For $s\geq s_c$ and $0\leq \alpha < s_c$, there exist $c, r>0$ and subset $\Omega_{\delta}\subset\Omega$ with $\PPP(\Omega^c_\dd)\leq e^{-1/\dd^r}$,  so that for all $\w\in \Omega_{\dd}$ and all $N_1\geq N_2 \geq N_3$, we can bound the integral:
$$\left|\int_0^\dd\int_{\mathbb{T}^d} \mathcal{N}(\widetilde{D}_{1}, \widetilde D_2, \widetilde R_3) \overline u_0 dxdt\right| \lesssim \dd^c (\frac{1}{N_2N_3})^c \|{u}_0\|_{Y^{-s}}\|{D}_1\|_{X^s}\|{D}_2\|_{X^{s}},$$
where $u_0$, $\widetilde{D}_{1}$, $\widetilde D_2$, and $\widetilde R_3$ is defined as (\ref{def:DR}).

\end{prop}
\begin{proof}
Let $P_{C_j}$ denote the family of Fourier
projections onto the cube $C_j$ of size $N_2$. We write $C_j \sim C_k$ if the sum set overlaps the Fourier support of $P_{≤2N_2}$.

\noindent\textit{(1) Case: $d=3$}

First, let's consider $\N_1 (\widetilde{D}_1, \widetilde{D}_2, \widetilde{R}_3 ) = \pm \widetilde{D}_1\widetilde{D}_2\widetilde{R}_3$.

By Corollary \ref{CorLp}, there exists $\Omega_\dd$ with $\PPP(\Omega_\dd^c)< e^{-1/\dd^r}$ and $c'>0$, such that for all $N_3$ and $\w \in \Omega_\dd$, we obtain that 
\begin{equation}\label{RandomLp10}
\|R_3\|_{L^{\frac{11}{2}}_{t,x}([0,\dd]\times\TTT^3)}\leq \delta^{c'} \frac{\log{N_3}}{N_3^{s_c-\alpha}}.
\end{equation}

By Lemma \ref{Strichartz}, Cauchy-Schwartz inequality and (\ref{RandomLp10}), 
\begin{align*}
& \left|\int_0^\dd\int_{\mathbb{T}^3}\overline{u}_0\widetilde{D}_1\widetilde{D}_2\widetilde{R}_3 dxdt\right|\\
 \leq& \sum_{C_j\sim C_k}\left|\int_0^\dd\int_{\mathbb{T}^3}(P_{C_j}\overline{u}_0)(P_{C_k}\widetilde{D}_1)\widetilde{D}_2\widetilde{R}_3 dxdt\right|\\
 \lesssim& \sum_{C_j\sim C_k} \|P_{C_j}u_0\|_{L_{t,x}^{\frac{11}{3}}}\|P_{C_k}D_1\|_{L_{t,x}^{\frac{11}{3}}}\|D_2\|_{L_{t,x}^{\frac{11}{3}}}\|R_3\|_{L_{t,x}^{\frac{11}{2}}}\\
\lesssim &  \sum_{C_j\sim C_k} \min \{N_0, N_2\}^{\frac{3}{22}}N_2^{\frac{6}{22}} \|P_{C_j}u_0\|_{Y^0}\|P_{C_k}D_1\|_{X^0}\|D_2\|_{X^0}\|R_3\|_{L_{t,x}^{\frac{11}{2}}}\\
\leq & \dd^{c'} \frac{\log N_3}{N_3^{s_c-\alpha}}\frac{1}{N_2^{\frac{1}{11}+s-s_c}} \sum_{C_j\sim C_k} \|P_{C_j}u_0\|_{Y^{-s}}\|P_{C_k}D_1\|_{X^s}\|D_2\|_{X^{s}}\\
\leq & \dd^c (\frac{1}{N_2N_3})^c  \|P_{C_j}u_0\|_{Y^{-s}}\|D_1\|_{X^s}\|D_2\|_{X^s},
 \end{align*}
 where $c = min(c', s_c-\alpha-\epsilon, \frac{1}{11})$.
 
Second, $\N_2(\widetilde{D}_1, \widetilde{D}_2, \widetilde{R}_3) = \pm \widetilde{D}_1 \int_{\mathbb{T}^3} \widetilde{D}_2 \widetilde{R}_3 dx$. We can bound $|\int_0^\dd\int_{\mathbb{T}^3} \mathcal{N}_2(\widetilde{D}_{1}, \widetilde D_2, \widetilde R_3) \overline u_0 dxdt|$ by $\sum_{C_j\sim C_k} \|P_{C_j}u_0\|_{L_{t,x}^{\frac{11}{3}}}\|P_{C_k}D_1\|_{L_{t,x}^{\frac{11}{3}}}\|D_2\|_{L_{t,x}^{\frac{11}{3}}}\|R_3\|_{L_{t,x}^{\frac{11}{2}}}$, using H\"{o}lder inequality. Then we can bound the second part via the same way.

\noindent\textit{(2) Case: $d\geq 4$}

First, let's consider $\N_1 (\widetilde{D}_1, \widetilde{D}_2, \widetilde{R}_3 ) = \pm \widetilde{D}_1\widetilde{D}_2\widetilde{R}_3$.

Set $3^{++}$ and $\infty^{--}$ as following:
\begin{equation}\label{coef:pq4d-}
\frac{1}{3^{++}} = \frac{1}{3} -\frac{1}{6(d+2)}, \quad \frac{1}{\infty^{--}} = \frac{1}{2(d+2)}.
\end{equation}

By Corollary \ref{CorLp}, there exists $\Omega_\dd$ with $\PPP(\Omega_\dd^c)< e^{-1/\dd^r}$ and $c'>0$, such that for all $N_3$ and $\w \in \Omega_\dd$, we obtain that 
\begin{equation}\label{RandomLp11}
\|R_3\|_{L^{\infty^{--}}_{t,x}([0,\dd]\times\TTT^d)}\leq \delta^{c'} \frac{\log{N_3}}{N_3^{s_c-\alpha}}.
\end{equation}

By Lemma \ref{Strichartz}, Cauchy-Schwartz inequality and (\ref{RandomLp11}), 
\begin{align*}
& \left|\int_0^\dd\int_{\mathbb{T}^d}\overline{u}_0\widetilde{D}_1\widetilde{D}_2\widetilde{R}_3 dxdt\right|\\
 \leq& \sum_{C_j\sim C_k}\left|\int_0^\dd\int_{\mathbb{T}^d}(P_{C_j}\overline{u}_0)(P_{C_k}\widetilde{D}_1)\widetilde{D}_2\widetilde{R}_3 dxdt\right|\\
 \lesssim& \sum_{C_j\sim C_k} \|P_{C_j}u_0\|_{L_{t,x}^{3^{++}}}\|P_{C_k}D_1\|_{L_{t,x}^{3^{++}}}\|D_2\|_{L_{t,x}^{3^{++}}}\|R_3\|_{L_{t,x}^{\infty^{--}}}\\
\lesssim &  \sum_{C_j\sim C_k} \min \{N_0, N_2\}^{\frac{d}{6}-\frac{1}{2}}N_2^{\frac{d}{3}-1} \|P_{C_j}u_0\|_{Y^0}\|P_{C_k}D_1\|_{X^0}\|D_2\|_{X^0}\|R_3\|_{L_{t,x}^{\infty^{--}}}\\
\leq & \dd^{c'} \frac{\log N_3}{N_3^{s_c-\alpha}}\frac{1}{N_2^{\frac{1}{2}+s-s_c}} \sum_{C_j\sim C_k} \|P_{C_j}u_0\|_{Y^{-s}}\|P_{C_k}D_1\|_{X^s}\|D_2\|_{X^{s}}\\
\leq & \dd^c (\frac{1}{N_2N_3})^c  \|P_{C_j}u_0\|_{Y^{-s}}\|D_1\|_{X^s}\|D_2\|_{X^s},
 \end{align*}
 where $c = min(c', s_c-\alpha-\epsilon, \frac{1}{2})$.
 
Second, $\N_2(\widetilde{D}_1, \widetilde{D}_2, \widetilde{R}_3) = \pm \widetilde{D}_1 \int_{\mathbb{T}^d} \widetilde{D}_2 \widetilde{R}_3 dx$. We can bound $|\int_0^\dd\int_{\mathbb{T}^d} \mathcal{N}_2(\widetilde{D}_{1}, \widetilde D_2, \widetilde R_3) \overline u_0 dxdt|$ by $\sum_{C_j\sim C_k} \|P_{C_j}u_0\|_{L_{t,x}^{3^{++}}}\|P_{C_k}D_1\|_{L_{t,x}^{3^{++}}}\|D_2\|_{L_{t,x}^{3^{++}}}\|R_3\|_{L_{t,x}^{\infty^{--}}}$, using H\"{o}lder inequality. Then we can bound the second part via the same way.

\end{proof}
\subsection*{\textbf{Case A (c)}}
We consider the case $(u_1, u_2, u_3)=(D_1, R_2, D_3)$.

\begin{prop}
Assume $N_i$, $i=0,1,2,3$, are dyadic numbers and $N_1\geq N_2 \geq N_3$, and $0 
\leq\dd \leq s_c$ . For $s\geq s_c$ and $0\leq \alpha <\frac{d}{6}$, there exists $c, r>0$ and subset $\Omega_{\delta}\subset\Omega$ with $\PPP(\Omega^c_\dd)\leq e^{-1/\dd^r}$,  so that for all $\w\in \Omega_{\dd}$ and all $N_1\geq N_2 \geq N_3$, we can bound the integral:
$$\left|\int_0^\dd\int_{\mathbb{T}^4} \mathcal{N}(\widetilde{D}_{1}, \widetilde R_2, \widetilde D_3) \overline u_0 dxdt\right| \lesssim \dd^c (\frac{1}{N_2N_3})^c \|{u}_0\|_{Y^{-s}}\|{D}_1\|_{X^s}\|{D}_3\|_{X^{s}},$$
where $u_0$, $\widetilde{D}_{1}$, $\widetilde R_2$, and $\widetilde D_3$ is defined as (\ref{def:DR}).
\end{prop}
\begin{proof}

Let $P_{C_j}$ denote the family of Fourier
projections onto the cube $C_j$ of size $N_2$. We write $Cj \sim Ck$ if the sum set overlaps the Fourier support of $P_{≤2N_2}$.

\noindent\textit{(1) Case: $d=3$}

First, let's consider $\N_1 (\widetilde{D}_1, \widetilde{R}_2, \widetilde{D}_3 ) = \pm \widetilde{D}_1\widetilde{R}_2\widetilde{D}_3$.

By Corollary \ref{CorLp}, there exists $\Omega_\dd$ with $\PPP(\Omega_\dd^c)< e^{-1/\dd^r}$ and $c'>0$, such that for all $N_2$ and $\w \in \Omega_\dd$, we obtain that 
\begin{equation}\label{RandomLp2}
\|R_2\|_{L^{\frac{11}{2}}_{t,x}([0,\dd]\times\TTT^3)}\leq \delta^{c'} \frac{\log{N_2}}{N_2^{s_c-\alpha}}.
\end{equation}

By Lemma \ref{Strichartz}, H\"{o}lder inequality and (\ref{RandomLp2}), 
\begin{align*}
& \left|\int_0^\dd\int_{\mathbb{T}^3}\overline{u}_0\widetilde{D}_1\widetilde{R}_2\widetilde{D}_3 dxdt\right|\\
 \leq& \sum_{C_j\sim C_k}\left|\int_0^\dd\int_{\mathbb{T}^3}(P_{C_j}\overline{u}_0)(P_{C_k}\widetilde{D}_1)\widetilde{R}_2\widetilde{D}_3 dxdt\right|\\
 \lesssim& \sum_{C_j\sim C_k} \|P_{C_j}u_0\|_{L_{t,x}^{\frac{11}{3}}}\|P_{C_k}D_1\|_{L_{t,x}^{\frac{11}{3}}}\|R_2\|_{L_{t,x}^{\frac{11}{2}}}\|D_3\|_{L_{t,x}^{\frac{11}{3}}}\\
\lesssim &  \min \{N_0, N_2\}^{\frac{3}{22}}N_2^{\frac{3}{22}}N_3^{\frac{3}{22}} \sum_{C_j\sim C_k} N_2^{3\epsilon} \|P_{C_j}u_0\|_{Y^0}\|P_{C_k}D_1\|_{X^0}\|D_3\|_{X^0}\|R_2\|_{L_{t,x}^{\frac{11}{2}}}\\
\leq & \dd^{c'} \frac{\log(N_2)}{N_2^{s_c-\alpha - \frac{3}{11}}N_3^{s-\frac{3}{22}}} \sum_{C_j\sim C_k}  \|P_{C_j}u_0\|_{Y^{-s}}\|P_{C_k}D_1\|_{X^s}\|D_3\|_{X^s}\\
\leq & \dd^c (\frac{1}{N_2N_3})^c  \|P_{C_j}u_0\|_{Y^{-s}}\|D_1\|_{X^s}\|D_3\|_{X^s},
 \end{align*}
 where $c = min(c', s_c-\alpha-\frac{3}{11}-\epsilon, s_c-\frac{3}{22})$.
 
Second, $\N_2(\widetilde{D}_1, \widetilde{R}_2, \widetilde{D}_3) = \pm \widetilde{D}_1 \int_{\mathbb{T}^3} \widetilde{R}_2 \widetilde{D}_3 dx$. We can bound $|\int_0^\dd\int_{\mathbb{T}^3} \mathcal{N}_2(\widetilde{D}_{1}, \widetilde R_2, \widetilde D_3) \overline u_0 dxdt|$ by $\sum_{C_j\sim C_k} \|P_{C_j}u_0\|_{L_{t,x}^{\frac{11}{3}}}\|P_{C_k}D_1\|_{L_{t,x}^{\frac{11}{3}}}\|D_3\|_{L_{t,x}^{\frac{11}{3}}}\|R_2\|_{L_{t,x}^{\frac{11}{2}}}$, using H\"{o}lder inequality. Then we can bound the second part via the same way.

\noindent\textit{(2) Case: $d\geq 4$}

First, let's consider $\N_1 (\widetilde{D}_1, \widetilde{R}_2, \widetilde{D}_3 ) = \pm \widetilde{D}_1\widetilde{R}_2\widetilde{D}_3$.

Set $3^{++}$ and $\infty^{--}$ as (\ref{coef:pq4d-}).
By Corollary \ref{CorLp}, there exists $\Omega_\dd$ with $\PPP(\Omega_\dd^c)< e^{-1/\dd^r}$ and $c'>0$, such that for all $N_2$ and $\w \in \Omega_\dd$, we obtain that 
\begin{equation}\label{RandomLp111}
\|R_2\|_{L^{\infty^{--}}_{t,x}([0,\dd]\times\TTT^d)}\leq \delta^{c'} \frac{\log{N_2}}{N_2^{s_c-\alpha}}.
\end{equation}

By Lemma \ref{Strichartz}, Cauchy-Schwartz inequality and (\ref{RandomLp111}), 
\begin{align*}
& \left|\int_0^\dd\int_{\mathbb{T}^d}\overline{u}_0\widetilde{D}_1\widetilde{R}_2\widetilde{D}_3 dxdt\right|\\
 \leq& \sum_{C_j\sim C_k}\left|\int_0^\dd\int_{\mathbb{T}^d}(P_{C_j}\overline{u}_0)(P_{C_k}\widetilde{D}_1)\widetilde{R}_2\widetilde{D}_3 dxdt\right|\\
 \lesssim& \sum_{C_j\sim C_k} \|P_{C_j}u_0\|_{L_{t,x}^{3^{++}}}\|P_{C_k}D_1\|_{L_{t,x}^{3^{++}}}\|D_3\|_{L_{t,x}^{3^{++}}}\|R_2\|_{L_{t,x}^{\infty^{--}}}\\
\lesssim &  \sum_{C_j\sim C_k} \min \{N_0, N_2\}^{\frac{d}{6}-\frac{1}{2}}N_2^{\frac{d}{6}-\frac{1}{2}}N_3^{\frac{d}{6}-\frac{1}{2}} \|P_{C_j}u_0\|_{Y^0}\|P_{C_k}D_1\|_{X^0}\|D_3\|_{X^0}\|R_2\|_{L_{t,x}^{\infty^{--}}}\\
\leq & \dd^{c'} \frac{\log N_2}{N_2^{\frac{d}{6}-\alpha}}\frac{1}{N_3^{\frac{d}{3}-\frac{1}{2}+s-s_c}} \sum_{C_j\sim C_k} \|P_{C_j}u_0\|_{Y^{-s}}\|P_{C_k}D_1\|_{X^s}\|D_3\|_{X^{s}}\\
\leq & \dd^c (\frac{1}{N_2N_3})^c  \|P_{C_j}u_0\|_{Y^{-s}}\|D_1\|_{X^s}\|D_3\|_{X^s},
 \end{align*}
 where $c = min(c', \frac{d}{6}-\alpha-\epsilon, \frac{d}{3}-\frac{1}{2})$.
 
Second, $\N_2(\widetilde{D}_1, \widetilde{D}_2, \widetilde{R}_3) = \pm \widetilde{D}_1 \int_{\mathbb{T}^d} \widetilde{D}_2 \widetilde{R}_3 dx$. We can bound $|\int_0^\dd\int_{\mathbb{T}^d} \mathcal{N}_2(\widetilde{D}_{1}, \widetilde D_2, \widetilde R_3) \overline u_0 dxdt|$ by $\sum_{C_j\sim C_k} \|P_{C_j}u_0\|_{L_{t,x}^{3^{++}}}\|P_{C_k}D_1\|_{L_{t,x}^{3^{++}}}\|D_2\|_{L_{t,x}^{3^{++}}}\|R_3\|_{L_{t,x}^{\infty^{--}}}$, using H\"{o}lder inequality. Then we can bound the second part via the same way.

\end{proof}

\subsection*{\textbf{Case A (d)}}
We consider the case $(u_1, u_2, u_3)=(D_1, R_2, R_3)$.
\begin{prop}
Assume $N_i$, $i=0,1,2,3$, are dyadic numbers and $N_1\geq N_2 \geq N_3$, and $0 
\leq\dd \leq 1$ . For $s\geq s_c$ and $0\leq \alpha < s_c$, there exist $c, r>0$ and subset $\Omega_{\delta}\subset\Omega$ with $\PPP(\Omega^c_\dd)\leq e^{-1/\dd^r}$,  so that for all $\w\in \Omega_{\dd}$ and all $N_1\geq N_2 \geq N_3$, we can bound the integral:
$$\left|\int_0^\dd\int_{\mathbb{T}^4} \mathcal{N}(\widetilde{D}_{1}, \widetilde R_2, \widetilde R_3) \overline u_0 dxdt\right| \lesssim \dd^c (\frac{1}{N_2N_3})^c \|{u}_0\|_{Y^{-s}}\|{D}_1\|_{X^s},$$
where $u_0$, $\widetilde{D}_{1}$, $\widetilde R_2$, and $\widetilde R_3$ is defined as (\ref{def:DR}).
\end{prop}
\begin{proof}
Let $P_{C_j}$ denote the family of Fourier
projections onto the cube $C_j$ of size $N_2$. We write $Cj \sim Ck$ if the sum set overlaps the Fourier support of $P_{≤2N_2}$.

First, let's consider $\N_1 (\widetilde{D}_1, \widetilde{R}_2, \widetilde{R}_3 ) = \pm \widetilde{D}_1\widetilde{R}_2\widetilde{R}_3$.

Set $p_c^+$, $q$ as 
\begin{equation}\label{coef:pc+}
\frac{1}{p_c^+} = \frac{d}{2(d+2)} -\epsilon \text{ and } \frac{1}{q} = \frac{1}{2} - \frac{1}{p_c^{+}}. 
\end{equation}
By Corollary \ref{CorLp}, there exists $\Omega_\dd$ with $\PPP(\Omega_\dd^c)< e^{-1/\dd^r}$ and $c'>0$, such that for all $N$ and $\w \in \Omega_\dd$, we obtain that 
\begin{equation}\label{RandomLp3}
\|P_N v_0^\w\|_{L^{q}_{t,x}([0,\dd]\times\TTT^4)}\leq \delta^{c'} \frac{\log{N}}{N^{s_c-\alpha}}.
\end{equation}

By Lemma \ref{Strichartz}, Cauchy-Schwartz inequality and (\ref{RandomLp3}), 
\begin{align*}
& \left|\int_0^\dd\int_{\mathbb{T}^4}\overline{u}_0\widetilde{D}_1\widetilde{R}_2\widetilde{R}_3 dxdt\right|\\
 \leq& \sum_{C_j\sim C_k}\left|\int_0^\dd\int_{\mathbb{T}^4}(P_{C_j}\overline{u}_0)(P_{C_k}\widetilde{D}_1)\widetilde{R}_2\widetilde{R}_3 dxdt\right|\\
 \lesssim& \sum_{C_j\sim C_k} \|P_{C_j}u_0\|_{L_{t,x}^{p_c^+}}\|P_{C_k}D_1\|_{L_{t,x}^{p_c^+}}\|R_2\|_{L_{t,x}^{q}}\|R_3\|_{L_{t,x}^{q}}\\
\lesssim &  \sum_{C_j\sim C_k} N_2^{2\epsilon} \|P_{C_j}u_0\|_{Y^0}\|P_{C_k}D_1\|_{X^0}\|R_2\|_{L_{t,x}^{q}}\|R_3\|_{L_{t,x}^{q}}\\
\leq & \dd^{2c'} \frac{\log N_2}{N_2^{s_c-\alpha}}\frac{\log N_3}{N_3^{s_c-\alpha}} N_2^{2\epsilon}\sum_{C_j\sim C_k}  \|P_{C_j}u_0\|_{Y^{-s}}\|P_{C_k}D_1\|_{X^s}\\
\leq & \dd^c (\frac{1}{N_2N_3})^c   \|P_{C_j}u_0\|_{Y^{-s}}\|u_1\|_{X^s},
 \end{align*}
 where $c = min(2c', s_c-\alpha-\epsilon)$.
 
Second, $\N_2(\widetilde{D}_1, \widetilde{R}_2, \widetilde{R}_3) = \pm \widetilde{D}_1 \int_{\mathbb{T}^4} \widetilde{R}_2 \widetilde{R}_3 dx$. We can bound $|\int_0^\dd\int_{\mathbb{T}^4} \mathcal{N}_2(\widetilde{D}_{1}, \widetilde R_2, \widetilde R_3) \overline u_0 dxdt|$ by $\sum_{C_j\sim C_k} \|P_{C_j}u_0\|_{L_{t,x}^{p_c^+}}\|P_{C_k}D_1\|_{L_{t,x}^{p_c^+}}\|R_2\|_{L_{t,x}^{q}}\|R_3\|_{L_{t,x}^{q}}$, using H\"{o}lder inequality. Then we can bound the second part via the same way.
\end{proof}

In \textbf{Case B}, the top frequency is random term, so that the approach in \textbf{Case A} fails. 
In the following proofs of subcases of \textbf{Case B}, it will suffice to focus on the frequencies satisfying $N_0\sim N_1\geq N_2$, since if $N_0< N_2 \sim N_1$, then \textbf{Case B} can be treated as \textbf{Case A} which the top frequency is deterministic term.

\subsection*{\textbf{Case B (a)}}
We consider the all random case $(u_1, u_2, u_3)=(R_1, R_2, R_3)$. 

\begin{prop}
Assume $N_i$, $i=0,1,2,3$, are dyadic numbers and for any $N_1, N_2, N_3$, satisfying $N_1\geq N_2 \geq N_3$, and $0 
\leq\dd \leq 1$ . For  $\alpha < \frac{1}{4}$ and $s_c \leq s < s_c + \frac{1}{4} -\alpha$, there exist $c, r>0$ and subset $\Omega_{\delta}\subset\Omega$ with $\PPP(\Omega^c_\dd)\leq e^{-1/\dd^r}$,  so that for all $\w\in \Omega_{\dd}$ and all $N_1\geq N_2 \geq N_3$, we can bound the integral:
$$\left|\int_0^\dd\int_{\mathbb{T}^d} \mathcal{N}(\widetilde{R}_{1}, \widetilde R_2, \widetilde R_3) \overline u_0 dxdt\right| \lesssim \dd^c (\frac{1}{N_1})^c \|{u}_0\|_{Y^{-s}},$$
where $\widetilde{R}_{1}$, $\widetilde R_2$, and $\widetilde R_3$ is defined as (\ref{def:DR}) and
only one of $\widetilde{R}_{i}$ can be $\overline{R_i}$.
\end{prop}
\begin{proof}
Let's suppose that $\widetilde{R}_1=\overline{R_1}$, $\widetilde{R}_2={R_2}$ and $\widetilde{R}_3={R_3}$, and the other cases are similar (we will also explain how to prove in the others in the following proof).

Define $S(n, m) := \{(n_1, n_2, n_3)\in \ZZZ^d\times\ZZZ^d\times\ZZZ^d : -n_1+n_2+n_3 = n,\  -|n_1|^2+|n_2|^2+|n_3|^2 = m,\  n_1\neq n_2,\ n_3,\ and\ n_i \sim N_i\}$ (For example, if we consider $\mathcal{N}(R_1, \overline{R_2}, R_3)$ case, then the corresponding $S(n, m) := \{(n_1, n_2, n_3)\in \ZZZ^d\times\ZZZ^d\times\ZZZ^d : n_1-n_2+n_3 = n,\  |n_1|^2-|n_2|^2+|n_3|^2 = m,\  n_2\neq n_1,\ n_3,\ and\ n_i \sim N_i\}$), where $m\in \ZZZ$ and $n \in \ZZZ^d$.
Then we have
\begin{align*}
&\N(\overline R_1, R_2, R_3) := {\J}_1 + {\J}_2\\
=&\sum_{n\in \ZZZ^d, m\in \ZZZ} e^{in\cdot x+itm}\sum_{S(n,m)}\frac{\overline{g_{n_1}(\w)}}{\langle n_1\rangle^{d-1-\alpha}}\frac{g_{n_2}(\w)}{\langle n_2\rangle^{d-1-\alpha}}\frac{g_{n_3}(\w)}{\langle n_3\rangle^{d-1-\alpha}}\\
+& \sum_{n\in \ZZZ^d, n\sim N_i,\ i=1,2,3} \frac{|g_n(\w)|^2g_{n}(\w)}{\langle n\rangle^{3d-3-3\alpha}} e^{in\cdot x+it|n|^2}
\end{align*}

\noindent\textbf{Step 1 a)} 
{First}, let's consider $\J_1$ term.
By Prop \ref{timelocalTransferPrinciple}, to estimate $|\int_0^\dd\int_{\mathbb{T}^d} \overline u_0\J_1 (\overline R_1, R_2, R_3) dxdt|$, we can first consider $u_0$ as a linear solution $\mathds{1}_J e^{it\Delta} \phi$ in any small interval $J\subset [0, \delta]$ and get the bound of $\left|\int_{J\times\TTT^d} \overline{P_{N_0} e^{it\Delta}\phi} \J_1(\overline R_1, R_2, R_3) \,dxdt\right|$. Suppose $\phi (x) = \sum_{n\in \ZZZ^d} a_n e^{in\cdot x}$ and $\mathds{1}_J (t) = \sum_{k\in \ZZZ} b_k e^{ikt}$.

\begin{align*}
&\left|\int_{J\times\TTT^d} \overline{P_{N_0} e^{it\Delta}\phi} \J_1(\overline R_1, R_2, R_3) \,dxdt\right|\\
=& \left|\sum_{n\in \ZZZ^d,\, |n|\sim N_0\atop{k\in \ZZZ,\, |k|\lesssim N_1^2}}\sum_{S(n, |n|^2+k)} b_k \overline{a_n} \frac{\overline{g_{n_1}(\w)}}{\langle n_1\rangle^{d-1-\alpha}}\frac{g_{n_2}(\w)}{\langle n_2\rangle^{d-1-\alpha}}\frac{g_{n_3}(\w)}{\langle n_3\rangle^{d-1-\alpha}}\right|
\end{align*}
then by Lemma \ref{lem:boundFourier}, we have that $\sum_{|k|\lesssim N^2_1} |b_k|\lesssim \log N_1$. So

\begin{align*}
& \left|\sum_{n\in \ZZZ^d,\, |n|\sim N_0\atop{k\in \ZZZ,\, |k|\lesssim N_1^2}}\sum_{S(n, |n|^2+k)} b_k \overline{a_n} \frac{\overline{g_{n_1}(\w)}}{\langle n_1\rangle^{d-1-\alpha}}\frac{g_{n_2}(\w)}{\langle n_2\rangle^{d-1-\alpha}}\frac{g_{n_3}(\w)}{\langle n_3\rangle^{d-1-\alpha}}\right|\\
\leq & \|P_{N_0} \phi\|_{L^2_x} \sum_{|k|\lesssim N^2_1} |b_k|\left( \sum_{n\in \ZZZ^d, |n|\sim N_0} \left|\sum_{S(n, |n|^2+k)} \frac{\overline{g_{n_1}(\w)}}{\langle n_1\rangle^{d-1-\alpha}}\frac{g_{n_2}(\w)}{\langle n_2\rangle^{d-1-\alpha}}\frac{g_{n_3}(\w)}{\langle n_3\rangle^{d-1-\alpha}} \right|^2\right)^{\frac{1}{2}}
\end{align*}

By Lemma \ref{MLD}, after choosing a subset $\Omega^1_\dd$ with $\PPP(\Omega^1_\dd)\lesssim e^{-\frac{1}{\dd^2}}$, and by Lattice counting lemma (Lemma \ref{Counting1}),
we obtain that
\begin{align*}
& \|P_{N_0} \phi\|_{L^2_x} \sum_{|k|\lesssim N^2_1} |b_k| \left( \sum_{n\in \ZZZ^d, |n|\sim N_0} \left|\sum_{S(n, |n|^2+k)} \frac{\overline{g_{n_1}(\w)}}{\langle n_1\rangle^{d-1-\alpha}}\frac{g_{n_2}(\w)}{\langle n_2\rangle^{d-1-\alpha}}\frac{g_{n_3}(\w)}{\langle n_3\rangle^{d-1-\alpha}} \right|^2\right)^{\frac{1}{2}}\\
\lesssim &\|P_{N_0} \phi\|_{L^2_x} \sum_{|k|\lesssim N^2_1} |b_k|  \frac{N_1^\epsilon}{N_1^{d-1-\alpha} N_2^{d-1-\alpha} N_3^{d-1-\alpha}} \\
&\times\left|\{(n, n_1, n_2, n_3)\in \ZZZ^d\times\ZZZ^d\times\ZZZ^d\times\ZZZ^d: (n_1, n_2, n_3)\in S(n, |n|^2+k)\}\right|^{\frac{1}{2}}\\
\leq& \frac{N_1^{\epsilon}}{N_1^{s_c+\frac{1}{2}-\alpha}N_2^{s_c-\alpha}N_3^{s_c-\alpha}} \|P_{N_0} \phi\|_{L^2_x}.
\end{align*}

\noindent\textbf{Step 1 b)}
{Second}, let's consider $\J_2$,
By Lemma \ref{Neps}, there exists a set $\Omega^2_\dd$ with $\PPP({\Omega^2_\dd}^c)<e^{-1/\dd^{\epsilon}}$, for all $\w\in \Omega^2_\dd$, we have $|g_n(\w)|\lesssim \frac{\log (\langle n\rangle + 1)}{\delta^\epsilon}$.
\begin{align*}
 \|\J_2\|^2_{L_{x,t}^2}=& \sum_{n\in \ZZZ^4, n\sim N_i,\ i=1,2,3} \left|\frac{|g_n(\w)|^2g_{n}(\w)}{\langle n\rangle^{3d-3-3\alpha}}\right|^2
 \lesssim  \dd^{-6\epsilon} \frac{1}{N_1^{5d-6\alpha-6\epsilon}}.
 \end{align*}

 If we choose $\Omega_\dd = \Omega_\dd^1 \cap \Omega_\dd^2$, then we obtain that
 \begin{align*}
 &\left|\int_{J\times\TTT^d} \overline{P_{N_0} e^{it\Delta}\phi} \mathcal{N}(\overline R_1, R_2, R_3) \,dxdt\right|\\
 \lesssim& \dd^{-\epsilon} \frac{N^\epsilon_1}{N_1^{s_c + \frac{1}{2}-\alpha}N_2^{s_c-\alpha}N_3^{s_c-\alpha}} \|P_{N_0}\phi\|_{L^2_x},
 \end{align*}
(For simplicity, we use $\epsilon$ vaguely as a constant which we can choose arbitrary small and $\epsilon's$ in the different inequalities don't have to be the exactly same.) 
 
\noindent\textbf{Step 2} If we set $1^+$ and $\infty^-$ satisfying $\frac{1}{1^+} =1 -\epsilon$ and $\frac{1}{\infty^-} = \frac{\epsilon}{3}$,
by H\"older inequality and Lemma \ref{LpLarge} after excluding a subset of probability $e^{-\frac{1}{\dd^c}}$, we have 
\begin{align*}
 &\left|\int_{J\times\TTT^d} \overline{P_{N_0} e^{it\Delta}\phi} \mathcal{N}_1(\overline R_1, R_2, R_3) \,dxdt\right|\\
 =&\left|\int_{J\times\TTT^d} \overline{P_{N_0} e^{it\Delta}\phi}\ \overline R_1 R_2 R_3 \,dxdt\right|\\
 \leq & \|P_{N_0}e^{it\Delta} \phi\|_{L^{1^+}_{t}L^{2}_x(J\times \TTT^d)}\|R_1\|_{L^{\infty^-}_{t}L^6_x}\|R_2\|_{L^{\infty^-}_{t}L^6_x}\|R_3\|_{L^{\infty^-}_{t}L^6_x}\\
 \lesssim & |J|^{1-\epsilon} \frac{N_1^\epsilon}{N_1^{s_c-\alpha}N_2^{s_c-\alpha}N_3^{s_c-\alpha}} \|P_{N_0} \phi\|_{L^2_x}.
\end{align*}
And also we have
\begin{align*}
&\left|\int_{J\times\TTT^d} \overline{P_{N_0} e^{it\Delta}\phi} \mathcal{N}_2(\overline R_1, R_2, R_3) \,dxdt\right|\\
\lesssim & \left|\int_{J}\left(\int_{\TTT^d} \overline{P_{N_0} e^{it\Delta}\phi}\ \overline{R_1}\,dx\right) \left(\int_{\TTT^d} R_2 R_3\,dx\right)dt\right|\\
\leq & \|P_{N_0}e^{it\Delta} \phi\|_{L^{1^+}_{t}L^2_x(J\times \TTT^d)}\|R_1\|_{L^{\infty^-}_{t}L^2_x}\|R_2\|_{L^{\infty^-}_{t}L^2_x}\|R_3\|_{L^{\infty^-}_{t}L^2_x}\\
\lesssim & |J|^{1-\epsilon}\|P_{N_0}e^{it\Delta} \phi\|_{L^{\infty}_{t}L^2_x(J\times \TTT^d)} \|R_1\|_{L^{\infty^-}_{t}L^2_x}\|R_2\|_{L^{\infty^-}_{t}L^2_x}\|R_3\|_{L^{\infty^-}_{t}L^2_x}\\
\lesssim & |J|^{1-\epsilon} \frac{N_1^\epsilon}{N_1^{s_c-\alpha}N_2^{s_c-\alpha}N_3^{s_c-\alpha}} \|P_{N_0} \phi\|_{L^2_x}.
\end{align*}

So we obtain that 
\[
\left|\int_{J\times\TTT^d} \overline{P_{N_0} e^{it\Delta}\phi} \mathcal{N}(\overline R_1, R_2, R_3) \,dxdt\right|\lesssim |J|^{1-\epsilon} \frac{N_1^\epsilon}{N_1^{s_c-\alpha}N_2^{s_c-\alpha}N_3^{s_c-\alpha}} \|P_{N_0} \phi\|_{L^2_x}.
\]

\noindent\textbf{Step 3}
Average the estimates in Step 1 and Step 2, we obtain that
\begin{align}
&\left|\int_{J\times\TTT^d} \overline{P_{N_0} e^{it\Delta}\phi} \mathcal{N}(\overline R_1, R_2, R_3) \,dxdt\right|\\
 \lesssim & |J|^{\frac{1}{2}} \dd^{-\epsilon} \frac{N_1^{\epsilon}}{N_1^{s_c + \frac{1}{4}-\alpha}N_2^{s_c-\alpha}N_3^{s_c-\alpha}} \|P_{N_0} \phi\|_{L^2_x}\label{eq:LocalTimeEstimate}
\end{align}

By the estimate (\ref{eq:LocalTimeEstimate}) in Step 3 and Lemma \ref{timelocalTransferPrinciple}, we hold that
\[
\left|\int_0^\dd\int_{\mathbb{T}^d} \mathcal{N}(\overline{R}_{1},  R_2, R_3) \overline u_0 dxdt\right| \lesssim \dd^{\frac{1}{2}} \frac{N_1^{\epsilon}}{N_1^{s_c + \frac{1}{4}-\alpha}N_2^{s_c-\alpha}N_3^{s_c-\alpha}} \|u_0\|_{U_\Delta^2}
\]

\noindent \textbf{Step 4}
Set $p_c^+$, $q$ as (\ref{coef:pc+}) and $\frac{1}{\infty^-} = {\epsilon}$. 
 Using Strichartz estimate (\ref{eq:Strichartz2}), we have 
\begin{align*}
&\left|\int_{[0,\dd]\times\TTT^d} \overline{u_0} \mathcal{N}(\overline R_1, R_2, R_3) \,dxdt\right|\\\
\lesssim & \|u_0\|_{L^{p_c^+}_{t,x}} \|R_1\|_{L^{p_c^+}_{t,x}} \|R_2\|_{L^q_{t,x}} \|R_3\|_{L^{q}_{t,x}}\\
\lesssim & \dd^{\frac{d+4}{2(d+2)}-\epsilon} \|u_0\|_{L^{p_c^+}_{t,x}} \|R_1\|_{L^{\infty^-}_{t,x}} \|R_2\|_{L^{\infty-}_{t,x}} \|R_3\|_{L^{\infty-}}\\
\lesssim& \dd^{\frac{d+4}{2(d+2)}} \frac{1}{N_1^{s_c-\alpha}N_2^{s_c-\alpha}N_3^{s_c-\alpha}} \|u_0\|_{U^{p_c^+}_\Delta}.\\
\end{align*}

By the interpolation lemma (Lemma \ref{Interpolation2}) and the embedding property (\ref{eq:Embedding}), we obtain that
\[
\left|\int_0^\dd\int_{\mathbb{T}^d} \mathcal{N}(\overline{R}_{1},  R_2, R_3) \overline u_0 dxdt\right| \lesssim \dd^{\frac{1}{2}} \frac{N_1^\epsilon}{N_1^{s_c + \frac{1}{4}-\alpha -s}N_2^{s_c-\alpha}N_3^{s_c-\alpha}} \|u_0\|_{Y^{-s}}
\]
Since $s<s_c + \frac{1}{4} -\alpha$, we hold the proposition.
\end{proof}

\subsection*{\textbf{Case B (b)}}
We consider the all case $(u_1, u_2, u_3)=(R_1, R_2, D_3)$.
\begin{prop}
Assume $N_i$, $i=0,1,2,3$, are dyadic numbers and $N_1\geq N_2 \geq N_3$, and $0 
\leq\dd \leq 1$ . For $s_c\leq s <s_c+\frac{1}{4}-\alpha$ and $0\leq \alpha < \frac{1}{4}$, there exist $c, r>0$ and subset $\Omega_{\delta}\subset\Omega$ with $\PPP(\Omega^c_\dd)\leq e^{-1/\dd^r}$,  so that for all $\w\in \Omega_{\dd}$ and all $N_1\geq N_2 \geq N_3$, we can bound the integral:
$$\left|\int_0^\dd\int_{\mathbb{T}^4} \mathcal{N}(\widetilde{R}_{1}, \widetilde R_2, \widetilde D_3) \overline u_0 dxdt\right| \lesssim \dd^c (\frac{1}{N_1})^c \|{u}_0\|_{Y^{-s}}\|D_3\|_{X^s},$$
 where $\widetilde{R}_{1}$, $\widetilde R_2$, and $\widetilde D_3$ is defined as (\ref{def:DR}) and only one of $\{\widetilde{R}_{1}, \widetilde R_2, \widetilde D_3\}$ can be the conjugate.
\end{prop}

\begin{proof}
Let's suppose that $\widetilde{R}_1=\overline{R_1}$, $\widetilde{R}_2={R_2}$ and $\widetilde{D}_3={D_3}$, and the other cases are similar.

Define $S_3(n, n_3, m) := \{(n_1, n_2) \in \ZZZ^d\times\ZZZ^d : -n_1+n_2+n_3 = n,\  -|n_1|^2+|n_2|^2+|n_3|^2 = m,\  n_1\neq n_2,\ n_3,\ \text{and }n_i \sim N_i\}$.
Then we have
\begin{align*}
&\N(\overline R_1, R_2, D_3) := {\J}_1 + {\J}_2\\
=&\sum_{n\in \ZZZ^d, m\in \ZZZ} e^{in\cdot x+itm}\sum_{S(n,m)}\frac{\overline{g_{n_1}(\w)}}{\langle n_1\rangle^{d-1-\alpha}}\frac{g_{n_2}(\w)}{\langle n_2\rangle^{d-1-\alpha}}\widehat{D_3}(n_3)\\
+& \sum_{n\in \ZZZ^d, n\sim N_i,\ i=1,2,3} \frac{|g_n(\w)|^2\widehat{D_3}(n_3)}{\langle n\rangle^{2d-2-2\alpha}} e^{in\cdot x+it|n|^2}
\end{align*}

\noindent\textbf{Step 1 a)} 
{First}, let's consider $\J_1$ term.
By Prop \ref{timelocalTransferPrinciple}, to estimate $|\int_0^\dd\int_{\mathbb{T}^d} \overline u_0\J_1 (\overline R_1, R_2, D_3) dxdt|$, we can first consider $u_0$ as a linear solution $\mathds{1}_J e^{it\Delta} \phi$ in any small interval $J\subset [0, \delta]$ and get the bound of $\left|\int_{J\times\TTT^d} \overline{P_{N_0} e^{it\Delta}\phi} \J_1(\overline R_1, R_2, P_{N_3} e^{it\Delta}\phi^{(3)}) \,dxdt\right|$. Suppose $\phi (x) = \sum_{n\in \ZZZ^d} a_ne^{in\cdot x}$, $\phi^{(3)} (x) = \sum_{n\in \ZZZ^d} a^{(3)}_n e^{in\cdot x}$ and $\mathds{1}_J (t) = \sum_{k\in \ZZZ} b_k e^{ikt}$.

\begin{align*}
&\left|\int_{J\times\TTT^d} \overline{P_{N_0} e^{it\Delta}\phi} \J_1(\overline R_1, R_2, P_{N_3} e^{it\Delta}\phi^{(3)}) \,dxdt\right|\\
=& \left|\sum_{n\in \ZZZ^d,\, |n|\sim N_0\atop{k\in \ZZZ,\, |k|\lesssim N_1^2}}\sum_{S(n, |n|^2+k)} b_k \overline{a_n} \frac{\overline{g_{n_1}(\w)}}{\langle n_1\rangle^{d-1-\alpha}}\frac{g_{n_2}(\w)}{\langle n_2\rangle^{d-1-\alpha}}a_{n_3}^{(3)}\right| 
\end{align*}
then by Lemma \ref{lem:boundFourier}, we have that $\sum_{|k|\lesssim N^2_1} |b_k|\lesssim \log N_1$. So

\begin{align*}
& \left|\sum_{n\in \ZZZ^d,\, |n|\sim N_0\atop{k\in \ZZZ,\, |k|\lesssim N_1^2}}\sum_{S(n, |n|^2+k)} b_k \overline{a_n} \frac{\overline{g_{n_1}(\w)}}{\langle n_1\rangle^{d-1-\alpha}}\frac{g_{n_2}(\w)}{\langle n_2\rangle^{d-1-\alpha}}a_{n_3}^{(3)}\right|\\
\leq & \|P_{N_0} \phi\|_{L^2_x} \|P_{N_3}\phi^{(3)}\|_{L^2_x}  \sum_{|k|\lesssim N^2_1} |b_k| \left(\sum_{n\in \ZZZ^d, |n|\sim N_0\atop n_3 \in \ZZZ^d, |n_3|\sim N_3} \left|\sum_{S_3(n, n_3, |n|^2+k)} \frac{\overline{g_{n_1}(\w)}}{\langle n_1\rangle^{d-1-\alpha}}\frac{g_{n_2}(\w)}{\langle n_2\rangle^{d-1-\alpha}} \right|^2\right)^{\frac{1}{2}}
\end{align*}

By Lemma \ref{MLD}, after choosing a subset $\Omega^1_\dd$ with $\PPP(\Omega^1_\dd)\lesssim e^{-\frac{1}{\dd^2}}$, and by Lattice counting lemma (Lemma \ref{Counting1}),
we obtain that
\begin{align*}
&  \|P_{N_0} \phi\|_{L^2_x} \|P_{N_3}\phi^{(3)}\|_{L^2_x}  \sum_{|k|\lesssim N^2_1} |b_k| \left(\sum_{n\in \ZZZ^d, |n|\sim N_0\atop n_3 in \ZZZ^d, |n_3|\sim N_3} \left|\sum_{S_3(n, n_3, |n|^2+k)} \frac{\overline{g_{n_1}(\w)}}{\langle n_1\rangle^{d-1-\alpha}}\frac{g_{n_2}(\w)}{\langle n_2\rangle^{d-1-\alpha}} \right|^2\right)^{\frac{1}{2}}\\
\lesssim &\|P_{N_0} \phi\|_{L^2_x} \|P_{N_3}\phi^{(3)}\|_{L^2_x} \sum_{|k|\lesssim N^2_1} |b_k|  \frac{1}{N_1^{d-1-\alpha} N_2^{d-1-\alpha}} \\
&\times\left|\{(n, n_1, n_2, n_3)\in \ZZZ^d\times\ZZZ^d\times\ZZZ^d\times\ZZZ^d: (n_1, n_2, n_3)\in S(n, |n|^2+k)\}\right|^{\frac{1}{2}}\\
\leq& \frac{N_1^\epsilon N_3^{\frac{d}{2}}}{N_1^{s_c+\frac{1}{2}-\alpha}N_2^{s_c-\alpha}} \|P_{N_0} \phi\|_{L^2_x}\|P_{N_3}\phi^{(3)}\|_{L^2_x}.
\end{align*}

\noindent\textbf{Step 1 b)} {Second},  consider $\J_2$.

\begin{align}
\label{eq:Bb1}
 &\left|\int_{J\times\TTT^d} \overline{P_{N_0} e^{it\Delta}\phi} \J_2(\overline R_1, R_2, P_{N_3} e^{it\Delta}\phi^{(3)}) \,dxdt\right|\\\label{eq:Bb2}
\leq & \left|\sum_{n\in \ZZZ^d,\, |n|\sim N_0\atop{k\in \ZZZ,\, |k|\lesssim N_1^2}} b_k \overline{a_n} \frac{\overline{g_{n}(\w)}}{\langle n\rangle^{d-1-\alpha}}\frac{g_{n}(\w)}{\langle n\rangle^{d-1-\alpha}}a_{n}^{(3)}\right|
\end{align}
By Lemma \ref{Neps}, there exists a set $\Omega^2_\dd$ with $\PPP({\Omega^2_\dd}^c)<e^{-1/\dd^{\epsilon}}$, for all $\w\in \Omega^2_\dd$, we have $|g_n(\w)|\lesssim\frac{\log (\langle n\rangle +1)}{\delta^\epsilon}$.
\begin{align}
(\ref{eq:Bb1})&\leq \frac{N_1^{2\epsilon}}{N_1^{2d-2-2\alpha}} \left|\sum_{n\in \ZZZ^d,\, |n|\sim N_0\atop{k\in \ZZZ,\, |k|\lesssim N_1^2}} b_k \overline{a_n} a_{n}^{(3)}\right|\\
&\lesssim  \frac{N_1^{2\epsilon}{ \log (N_1)}}{N_1^{2d-2-2\alpha}} \|P_{N_0} \phi\|_{L^2_x}\|P_{N_3}\phi^{(3)}\|_{L^2_x}
 \end{align}

 If we choose $\Omega_\dd = \Omega_\dd^1 \cap \Omega_\dd^2$, then we obtain that
 \begin{align*}
 &\left|\int_{J\times\TTT^d} \overline{P_{N_0} e^{it\Delta}\phi} \mathcal{N}(\overline R_1, R_2, P_{N_3} e^{it\Delta}\phi^{(3)}) \,dxdt\right|\\
 \lesssim& \frac{N_1^\epsilon N_3^{\frac{d}{2}}}{N_1^{s_c + \frac{1}{2}-\alpha}N_2^{s_c-\alpha}} \|P_{N_0}\phi\|_{L^2_x} \|P_{N_3} \phi^{(3)}\|_{L^2_x},
 \end{align*}
\noindent\textbf{Step 2} If we set $2^+$, $\infty^-$, $1^+$ and $q$ satisfying $\frac{1}{2^+} = \frac{1}{2}-\epsilon$, $\frac{2}{\infty^-} + \frac{1}{2^+} = \frac{1}{2}$, $\frac{1}{1^+} + \frac{2}{\infty^-} = 1$, and $\frac{2}{q} + \frac{d}{2^+} = \frac{d}{2}$.
By H\"older inequality and Lemma \ref{LpLarge} after excluding a subset of probability $e^{-\frac{1}{\dd^c}}$, we have 
\begin{align*}
 &\left|\int_{J\times\TTT^d} \overline{P_{N_0} e^{it\Delta}\phi} \mathcal{N}_1(\overline R_1, R_2, P_{N_3} e^{it\Delta}\phi^{(3)}) \,dxdt\right|\\
 =&\left|\int_{J\times\TTT^d} \overline{P_{N_0} e^{it\Delta}\phi}\ \overline R_1 R_2 P_{N_3} e^{it\Delta}\phi^{(3)} \,dxdt\right|\\
 \leq & \|P_{N_0}e^{it\Delta} \phi\|_{L^{\infty}_{t}L^{2}_x(J\times \TTT^d)}\|R_1\|_{L^{\infty^-}_{t,x}}\|R_2\|_{L^{\infty^-}_{t,x}}\|P_{N_3} e^{it\Delta}\phi^{(3)}\|_{L^{1^+}_{t}L^{2^+}_x}\\
 \lesssim & |J|^{1-\epsilon} \frac{N_3^\epsilon}{N_1^{s_c-\alpha}N_2^{s_c-\alpha}} \|P_{N_0} \phi\|_{L^2_x}\|P_{N_3} \phi^{(3)}\|_{L^2_x}.
\end{align*}
And also we have
\begin{align*}
&\left|\int_{J\times\TTT^d} \overline{P_{N_0} e^{it\Delta}\phi} \mathcal{N}_2(\overline R_1, R_2, P_{N_3} e^{it\Delta}\phi^{(3)}) \,dxdt\right|\\
\lesssim & \left|\int_{J}\left(\int_{\TTT^d} \overline{P_{N_0} e^{it\Delta}\phi}\ \overline{R_1}\,dx\right) \left(\int_{\TTT^d} R_2 P_{N_3} e^{it\Delta}\phi^{(3)}\,dx\right)dt\right|\\
 \leq & \|P_{N_0}e^{it\Delta} \phi\|_{L^{\infty}_{t}L^{2}_x(J\times \TTT^d)}\|R_1\|_{L^{\infty^-}_{t}L^{2}_x}\|R_2\|_{L^{\infty^-}_{t}L^2_x}\|P_{N_3} e^{it\Delta}\phi^{(3)}\|_{L^{1^+}_{t}L^{2}_x}\\
 \lesssim & |J|^{1-\epsilon} \frac{1}{N_1^{s_c-\alpha}N_2^{s_c-\alpha}} \|P_{N_0} \phi\|_{L^2_x}\|P_{N_3} \phi^{(3)}\|_{L^2_x}.
\end{align*}

So we obtain that 
\[
\left|\int_{J\times\TTT^d} \overline{P_{N_0} e^{it\Delta}\phi} \mathcal{N}(\overline R_1, R_2, P_{N_3} e^{it\Delta}\phi^{(3)}) \,dxdt\right|
\lesssim |J|^{1-\epsilon} \frac{N_3^\epsilon}{N_1^{s_c-\alpha}N_2^{s_c-\alpha}} \|P_{N_0} \phi\|_{L^2_x}\|P_{N_3} \phi^{(3)}\|_{L^2_x}.
\]
\noindent\textbf{Step 3}
Average the estimates in Step 1 and Step 2, we obtain that
\begin{align}
&\left|\int_{J\times\TTT^d} \overline{P_{N_0} e^{it\Delta}\phi} \mathcal{N}(\overline R_1, R_2, P_{N_3} e^{it\Delta}\phi^{(3)}) \,dxdt\right|\\
 \lesssim & |J|^{\frac{1}{2}} \frac{N_1^{\epsilon}N_3^{\frac{d}{4}}}{N_1^{s_c + \frac{1}{4}-\alpha}N_2^{s_c-\alpha}} \|P_{N_0} \phi\|_{L^2_x}\|P_{N_3} \phi^{(3)}\|_{L^2_x}.\label{eq:LocalTimeEstimate2}
\end{align}

By the estimate (\ref{eq:LocalTimeEstimate2}) in Step 3 and Lemma \ref{timelocalTransferPrinciple}, we hold that
\[
\left|\int_0^\dd\int_{\mathbb{T}^d} \mathcal{N}(\overline{R}_{1},  R_2, D_3) \overline u_0 dxdt\right| \lesssim \dd^{\frac{1}{2}} \frac{N_1^{\epsilon}N_3^{\frac{d}{4}}}{N_1^{s_c + \frac{1}{4}-\alpha}N_2^{s_c-\alpha}} \|u_0\|_{U_\Delta^2}\|D_3\|_{U_\Delta^2}
\]
\noindent \textbf{Step 4}
Set $p_c^+$, $q$ as (\ref{coef:pc+}) and $\frac{1}{\infty^-} = {\epsilon}$
Replacing $\overline{P_{N_0} e^{it\Delta}\phi}$ by $\overline{u_0}$, following the similar idea, and using Strichartz estimate (\ref{eq:Strichartz2}), we have 
\begin{align*}
&\left|\int_{[0,\dd]\times\TTT^d} \overline{u_0} \mathcal{N}(\overline R_1, R_2, D_3) \,dxdt\right|
\lesssim  \|u_0\|_{L^{p_c^+}_{t,x}} \|R_1\|_{L^{q}_{t,x}} \|R_2\|_{L^q_{t,x}} \|D_3\|_{L^{p_c^+}_{t,x}}\\
\lesssim & \dd^{\frac{4}{2(d+2)}-\epsilon} \|u_0\|_{L^{p_c^+}_{t,x}} \|R_1\|_{L^{\infty^-}_{t,x}} \|R_2\|_{L^{\infty-}_{t,x}} \|D_3\|_{L^{p_c^+}}\\
\lesssim& \dd^{\frac{4}{2(d+2)}-\epsilon} \frac{1}{N_1^{s_c-\alpha-\epsilon}N_2^{s_c-\alpha}} \|u_0\|_{U^{p_c^+}_\Delta}\|D_3\|_{U^{p_c^+}_\Delta}.\\
\end{align*}

By Lemma \ref{Interpolation2} and the embedding property (\ref{eq:Embedding}), we obtain that
\[
\left|\int_0^\dd\int_{\mathbb{T}^d} \mathcal{N}(\overline{R}_{1},  R_2, D_3) \overline u_0 dxdt\right| \lesssim \dd^{\frac{1}{2}} \frac{N_1^{\epsilon}N_3^{\frac{d}{4}}}{N_1^{s_c + \frac{1}{4}-\alpha -s}N_2^{s_c-\alpha} N_3^{s}} \|u_0\|_{Y^{-s}} \|D_3\|_{X^s}
\]
Since $s<s_c + \frac{1}{4} -\alpha$, we hold the proposition.\end{proof}
\subsection*{\textbf{Case B (c)}}
We consider the all case $(u_1, u_2, u_3)=(R_1, D_2, R_3)$. By the similar approach with \textbf{Case B (b)}, we can hold following property:
\begin{prop}
Assume $N_i$, $i=0,1,2,3$, are dyadic numbers and $N_1\geq N_2 \geq N_3$, and $0 
\leq\dd \leq 1$. For $s_c\leq s <s_c + \frac{1}{6} -\alpha$ and $0\leq \alpha < \frac{1}{6}$, there exist $c, r>0$ and subset $\Omega_{\delta}\subset\Omega$ with $\PPP(\Omega^c_\dd)\leq e^{-1/\dd^r}$,  so that for all $\w\in \Omega_{\dd}$ and all $N_1\geq N_2 \geq N_3$, we can bound the integral:
$$\left|\int_0^\dd\int_{\mathbb{T}^4} \mathcal{N}(\widetilde{R}_{1}, \widetilde D_2, \widetilde R_3) \overline u_0 dxdt\right| \lesssim \dd^c (\frac{1}{N_1})^c \|{u}_0\|_{Y^{-s}} \|D_2\|_{X^{s}},$$
where $\widetilde{R}_{1}$, $\widetilde D_2$, and $\widetilde R_3$ is defined as (\ref{def:DR}) and only one of $\{\widetilde{R}_{1}, \widetilde D_2, \widetilde R_3\}$ can be the conjugate.
\end{prop}
\begin{proof}
For $d\geq 4$.

Following the \textbf{Case B (b)}, by choosing a subset $\Omega_\dd\subset \Omega$ with $\PPP(\Omega_\dd^c)<e^{-1/\dd^r}$, for $\w\in\Omega_\dd$, we have similar estimate:
\[\left|\int_0^\dd\int_{\mathbb{T}^d} \N(\widetilde R_1, \widetilde D_2, \widetilde R_3) \overline u_0 dxdt\right|\lesssim \dd^{\frac{1}{2}} \frac{N_1^{\epsilon}N_2^{\frac{d}{4}}}{N_1^{s_c + \frac{1}{4}-\alpha -s}N_3^{s_c-\alpha} N_2^{s}} \|u_0\|_{Y^{-s}} \|D_2\|_{X^s},
\]
since $s<s_c + \frac{1}{4} -\alpha$, the proposition holds.

For $d = 3$.
Following the \textbf{Case B (b)},
in {Step 3}, we average the estimates in Step 1 and Step 2 with different weights, we have that
\[
\left|\int_0^\dd\int_{\mathbb{T}^3} \N(\widetilde R_1, \widetilde D_2, \widetilde R_3) \overline u_0 dxdt\right|\lesssim \dd^{\frac{1}{2}} \frac{N_1^{\epsilon}N_2^{\frac{1}{2}}}{N_1^{s_c + \frac{1}{6}-\alpha -s}N_3^{s_c-\alpha} N_2^{s}} \|u_0\|_{Y^{-s}} \|D_2\|_{X^s},
\]

since $s_c\leq s<s_c + \frac{1}{6} -\alpha$, and $s_c = \frac{1}{2}$, the proposition holds.

\end{proof}
\subsection*{\textbf{Case B (d)}}
We consider the case $(u_1, u_2, u_3)=(R_1, D_2, D_3)$.
\begin{prop}
Assume $N_i$, $i=0,1,2,3$, are dyadic numbers and $N_1\geq N_2 \geq N_3$, and $0 
\leq\dd \leq 1$. For $d\geq 3$,  \[s_r(d) = \begin{cases}
\frac{1}{7}& d= 3\\
\frac{4}{19}& d= 4\\
\frac{1}{4} & d\geq 5.
\end{cases}\]
For $s_c \leq s <s_c + s_r(d)-\alpha$ and $0\leq \alpha < s_r (d)$, there exist $c, r>0$ and subset $\Omega_{\delta}\subset\Omega$ with $\PPP(\Omega^c_\dd)\leq e^{-1/\dd^r}$,  so that for all $\w\in \Omega_{\dd}$ and all $N_1\geq N_2 \geq N_3$, we can bound the integral:
$$\left|\int_0^\dd\int_{\mathbb{T}^d} \mathcal{N}(\widetilde{R}_{1}, \widetilde D_2, \widetilde D_3) \overline u_0 dxdt\right| \lesssim \dd^c (\frac{1}{N_1})^c \|{u}_0\|_{Y^{-s}} \|D_2\|_{X^{s}}\|D_3\|_{X^{s}},$$
where $\widetilde{R}_{1}$, $\widetilde D_2$, and $\widetilde D_3$ is defined as (\ref{def:DR}) and only one of $\{\widetilde{R}_{1}, \widetilde D_2, \widetilde D_3\}$ can be the conjugate.
\end{prop}
\begin{proof}
Let's suppose that $\widetilde{R}_1={R_1}$, $\widetilde{D}_2=\overline{D_2}$ and $\widetilde{D}_3={D_3}$, and the other cases are similar.

Define $S_{2, 3}(n, n_2, n_3, m) := \{n_1 \in \ZZZ^d : n_1-n_2+n_3 = n,\  |n_1|^2-|n_2|^2+|n_3|^2 = m,\  n_2\neq n_1,\ n_3,\ \text{and}\ n_1 \sim N_1\}$.
Then we have
\begin{align*}
&\N( R_1, \overline D_2, D_3) := {\J}_1 + {\J}_2\\
=&\sum_{n\in \ZZZ^d, m\in \ZZZ} e^{in\cdot x+itm}\sum_{S(n,m)}\frac{{g_{n_1}(\w)}}{\langle n_1\rangle^{d-1-\alpha}}\widehat{\overline D_2}(n_2)\widehat{D_3}(n_3)\\
+& \sum_{n\in \ZZZ^d, n\sim N_i,\ i=1,2,3} \frac{{g_n(\w)}\widehat{\overline D_2}(n_2)\widehat{D_3}(n_3)}{\langle n\rangle^{d-1-\alpha}} e^{in\cdot x+it|n|^2}
\end{align*}

\noindent\textbf{Step 1 a)} 
{First}, let's consider $\J_1$ term.
By Proposition \ref{timelocalTransferPrinciple}, to estimate $|\int_0^\dd\int_{\mathbb{T}^d} \overline u_0\J_1 ( R_1, \overline D_2, D_3) dxdt|$, we can first consider $u_0$ as a linear solution $\mathds{1}_J e^{it\Delta} \phi$ in any small interval $J\subset [0, \delta]$ and get the bound of $\left|\int_{J\times\TTT^d} \overline{P_{N_0} e^{it\Delta}\phi} \J_1( R_1, \overline{P_{N_2} e^{it\Delta}\phi^{(2)}}, P_{N_3} e^{it\Delta}\phi^{(3)}) \,dxdt\right|$. Suppose $\phi (x) = \sum_{n\in \ZZZ^d} a_ne^{in\cdot x}$, $\phi^{i} (x) = \sum_{n\in \ZZZ^d} a^{(i)}_n e^{in\cdot x}$ for $i = 2, 3$ and $\mathds{1}_J (t) = \sum_{k\in \ZZZ} b_k e^{ikt}$.

\begin{align*}
&\left|\int_{J\times\TTT^d} \overline{P_{N_0} e^{it\Delta}\phi} \J_1( R_1,  \overline{P_{N_2} e^{it\Delta}\phi^{(2)}}, P_{N_3} e^{it\Delta}\phi^{(3)}) \,dxdt\right|\\
=& \left|\sum_{n\in \ZZZ^d,\, |n|\sim N_0\atop{k\in \ZZZ,\, |k|\lesssim N_1^2}}\sum_{S(n, |n|^2+k)} b_k \overline{a_n} \frac{{g_{n_1}(\w)}}{\langle n_1\rangle^{d-1-\alpha}}\overline{a_{n_2}^{(2)}}a_{n_3}^{(3)}\right|\\
\leq & \|P_{N_0} \phi\|_{L^2_x} \|P_{N_3}\phi^{(3)}\|_{L^2_x}  \sum_{|k|\lesssim N^2_1} |b_k| \left(\sum_{n\in \ZZZ^d, |n|\sim N_0\atop n_3\in \ZZZ^d |n_3|\sim N_3} \left|\sum_{S_{3}(n, n_3, |n|^2+k)} \frac{{g_{n_1}(\w)}}{\langle n_1\rangle^{d-1-\alpha}} \overline{a_{n_2}^{(2)}} \right|^2\right)^{\frac{1}{2}}
\end{align*}

Next, fix $k$. Let's focus on
\begin{equation}\label{GoalOfMatrix}
\sum_{n\in \ZZZ^d, |n|\sim N_0\atop n_3\in \ZZZ^d |n_3|\sim N_3} \left|\sum_{S_{3}(n, n_3, |n|^2+k)} \frac{{g_{n_1}(\w)}}{\langle n_1\rangle^{d-1-\alpha}} \overline{a_{n_2}^{(2)}} \right|^2.
\end{equation}

To bound (\ref{GoalOfMatrix}), we use the matrix $\mathscr{G}^* \mathscr{G}$ argument in Bourgain's paper \cite{bourgain1996periodic2d} as follows.

Fix $n_3$ and $|n_3|\sim N_3$. Define 
\[\mathscr{G} = \mathscr{G}_\omega = (\sigma_{n, n_2})_{ |n|<N_0, |n_2|<N_2, \atop n\neq n_3}\]
where
\[
\sigma_{n, n_2} =
\begin{cases}
\frac{1}{N_1^{d-1-\alpha}} {g_{n+n_2-n_3} (\omega)}& \text{if } 2\langle n-n_3, n_2-n_3\rangle = k\\
0& \text{otherwise}
\end{cases}
\]

Then (\ref{GoalOfMatrix}) is bounded by $N_3^{d} \|\mathscr{G}_\omega^* \mathscr{G}_\omega\|^{\frac{1}{2}}$ and by {Lemma} \ref{lem:matrix}, we obtain that
\begin{equation}\label{eq:MatrixTwoTerms}
\|\mathscr{G}^*\mathscr{G}\|\leq \max_n \left(\sum_{n_2} |\sigma_{n, n_2}|^2\right) + \left( \sum_{n\neq n'} \left|\sum_{n_2} \sigma_{n, n_2} \overline{\sigma_{n',n_2}}\right|^2 \right)^{\frac{1}{2}}.
\end{equation}
Using {Lemma \ref{Neps}}, the first term in (\ref{eq:MatrixTwoTerms}) is bounded as follows,
\begin{align}\label{eq:MatrixEstimate1}
\sum_{n_2} \left|\frac{1}{N_1^{d-1-\alpha}} g_{n+n_2-n_3}(\omega)\right|^2
\leq \frac{N_2^d}{N_1^{2(d-1-\alpha)-\epsilon}}\leq \frac{N_2}{N_1^{2s_c + 1 -2\alpha -\epsilon}}.
\end{align}

Then we will show that the second term in (\ref{eq:MatrixTwoTerms}) is bounded as follows
\begin{equation}\label{eq:MatrixEstimate2}
\left( \sum_{n\neq n'} \left|\sum_{n_2} \sigma_{n, n_2} \overline{\sigma_{n',n_2}}\right|^2 \right)^{\frac{1}{2}}
\leq {N_1^{-2s_c -1 +2\alpha +\epsilon}}{N_2^{\frac{d-1}{2}}}
\end{equation}
Indeed, write
\begin{equation}\label{eq:MatrixEstimate3}
\sum_{n\neq n'} \left|\sum_{n_2} \sigma_{n, n_2} \overline{\sigma_{n',n_2}}\right|^2 =
\frac{1}{N_1^{4(d-1-\alpha)}} \sum_{n\neq n'} \left|\sum_{n_2} g_{n+n_2-n_3}(\omega) \overline{g_{n'+n_2-n_3}}(\omega)\right|^2.
\end{equation}

Then we use  Lemma \ref{MLD}, there exists a set $\Omega^1_\dd$ with $\PPP({\Omega^1_\dd}^c)<e^{-1/\dd^{r}}$, for all $\w\in \Omega^1_\dd$, (\ref{eq:MatrixEstimate3}) can be bounded by
\begin{align}
\label{eq:MatrixEstimate4}\begin{split}
\frac{1}{N_1^{4(d-1-\alpha)}}&
|\{ (n, n', n_2): n\neq n', 2\langle n-n_3, n_2-n_3\rangle = k\\
 & 2\langle n'-n_3, n_2-n_3\rangle = k\}|\end{split}
\end{align}

To bound the number of the elements in $\{ (n, n', n_2): n\neq n', 2\langle n-n_3, n_2-n_3\rangle = k,\ 2\langle n'-n_3, n_2-n_3\rangle = k\}$, first we can count the number of pair $(n, n_2)$ and
by Lemma \ref{Counting2}, it is bounded by $N_1^{d-1+\epsilon} N_2^{d-1}$. And then the number of possible $n'$ is bounded $N_1^{d-1}$ by Lemma \ref{Counting1}. So the size of the upper lattice set is bounded by $N_1^{2(d-1)+\epsilon}N_2^{d-1}$, and hence we hold (\ref{eq:MatrixEstimate2}).

By the estimates (\ref{eq:MatrixEstimate1}) and (\ref{eq:MatrixEstimate2}), we obtain that
\begin{align}\label{eq:BdStep0}
 \begin{split}
 &\left|\int_{J\times\TTT^d} \overline{P_{N_0} e^{it\Delta}\phi} \mathcal{J}_1(\overline R_1, P_{N_2} e^{it\Delta}\phi^{(2)}, P_{N_3} e^{it\Delta}\phi^{(3)}) \,dxdt\right|\\
 \lesssim& \frac{N_1^\epsilon N_2^{\frac{d-1}{4}} N_3^{\frac{d}{2}}}{N_1^{s_c+\frac{1}{2}-\alpha}} \|P_{N_0} \phi\|_{L^2_x}\|P_{N_2}\phi^{(2)}\|_{L^2_x} \|P_{N_3}\phi^{(3)}\|_{L^2_x},
 \end{split}
 \end{align}

\noindent\textbf{Step 1 b)} 
{Second}, let's consider $\J_2$.
By Lemma \ref{Neps}, there exists a set $\Omega^2_\dd$ with $\PPP({\Omega^2_\dd}^c)<e^{-1/\dd^{2\epsilon/3}}$, for all $\w\in \Omega^2_\dd$, we have $|g_n(\w)|\lesssim\frac{\langle n\rangle^\epsilon}{\delta^\epsilon}$.
\begin{align*}
 &\left|\int_{J\times\TTT^d} \overline{P_{N_0} e^{it\Delta}\phi} \J_2(\overline R_1, P_{N_2} e^{it\Delta}\phi^{(2)}, P_{N_3} e^{it\Delta}\phi^{(3)}) \,dxdt\right|\\ 
\leq & \left|\sum_{n\in \ZZZ^d,\, |n|\sim N_0\atop{k\in \ZZZ,\, |k|\lesssim N_1^2}} b_k \overline{a_n} \frac{\overline{g_{n}(\w)}}{\langle n\rangle^{d-1-\alpha}}a_{n}^{(2)}a_{n}^{(3)}\right|
\leq  \frac{N_1^{2\epsilon}}{N_1^{d-1-\alpha}} \left|\sum_{n\in \ZZZ^d,\, |n|\sim N_0\atop{k\in \ZZZ,\, |k|\lesssim N_1^2}} b_k \overline{a_n} a_{n}^{(2)} a_{n}^{(3)}\right|\\
\lesssim & \frac{N_1^{2\epsilon}{ \log (N_1)}}{N_1^{d-1-\alpha}} \|P_{N_0} \phi\|_{L^2_x}\|P_{N_2}\phi^{(2)}\|_{L^2_x}\|P_{N_3}\phi^{(3)}\|_{L^1_x(\TTT^d)}\\
\lesssim &\frac{N_1^{2\epsilon}{ \log (N_1)}}{N_1^{d-1-\alpha}} \|P_{N_0} \phi\|_{L^2_x}\|P_{N_2}\phi^{(2)}\|_{L^2_x}\|P_{N_3}\phi^{(3)}\|_{L^2_x(\TTT^d)}.
 \end{align*}

 If we choose $\Omega_\dd = \Omega_\dd^1 \cap \Omega_\dd^2$, then we obtain that
 \begin{align}\label{eq:BdStep1}
 \begin{split}
 &\left|\int_{J\times\TTT^d} \overline{P_{N_0} e^{it\Delta}\phi} \mathcal{N}(\overline R_1, P_{N_2} e^{it\Delta}\phi^{(2)}, P_{N_3} e^{it\Delta}\phi^{(3)}) \,dxdt\right|\\
 \lesssim& \frac{N_1^\epsilon N_2^{\frac{d-1}{4}} N_3^{\frac{d}{2}}}{N_1^{s_c+\frac{1}{2}-\alpha}} \|P_{N_0} \phi\|_{L^2_x}\|P_{N_2}\phi^{(2)}\|_{L^2_x} \|P_{N_3}\phi^{(3)}\|_{L^2_x},
 \end{split}
 \end{align}
 
\noindent\textbf{Step 2} If we set $2^+$, $\infty^-$, $1^+$ and $q$ satisfying $\frac{1}{2^+} = \frac{1}{2}-\epsilon$, $\frac{2}{\infty^-} + \frac{1}{2^+} = \frac{1}{2}$, $\frac{1}{1^+} + \frac{2}{\infty^-} = 1$, and $\frac{2}{q} + \frac{d}{2^+} = \frac{d}{2}$.
By H\"older inequality and Lemma \ref{LpLarge} after excluding a subset of probability $e^{-\frac{1}{\dd^c}}$, we have 
\begin{align*}
 &\left|\int_{J\times\TTT^d} \overline{P_{N_0} e^{it\Delta}\phi} \mathcal{N}_1
 (\overline R_1, P_{N_2} e^{it\Delta}\phi^{(2)}, P_{N_3} e^{it\Delta}\phi^{(3)})\,dxdt\right|\\
 =&\left|\int_{J\times\TTT^d} \overline{P_{N_0} e^{it\Delta}\phi}\ \overline R_1 P_{N_2} e^{it\Delta}\phi^{(2)}  P_{N_3} e^{it\Delta}\phi^{(3)} \,dxdt\right|\\
 \leq & \|P_{N_0}e^{it\Delta} \phi\|_{L^{\infty}_{t}L^{2}_x(J\times \TTT^d)}\|R_1\|_{L^{\infty^-}_{t,x}}\|P_{N_2} e^{it\Delta}\phi^{(2)}\|_{L^{1^+}_{t}L^{2^+}_x} \|P_{N_3} e^{it\Delta}\phi^{(3)}\|_{L^{\infty^-}_{t,x}}\\
 \lesssim & |J|^{1-\epsilon} \frac{N_2^\epsilon N_3^{\frac{d}{2}-\epsilon}}{N_1^{s_c-\alpha}} \|P_{N_0} \phi\|_{L^2_x}\|P_{N_2} \phi^{(2)}\|_{L^2_x}\|P_{N_3} \phi^{(3)}\|_{L^2_x}.
\end{align*}
And also we have
\begin{align*}
&\left|\int_{J\times\TTT^d} \overline{P_{N_0} e^{it\Delta}\phi} \mathcal{N}_2(\overline R_1, P_{N_2} e^{it\Delta}\phi^{(2)}, P_{N_3} e^{it\Delta}\phi^{(3)}) \,dxdt\right|\\
\lesssim & \left|\int_{J}\left(\int_{\TTT^d} \overline{P_{N_0} e^{it\Delta}\phi}\ \overline{R_1}\,dx\right) \left(\int_{\TTT^d} P_{N_2} e^{it\Delta}\phi^{(2)} P_{N_3} e^{it\Delta}\phi^{(3)}\,dx\right)dt\right|\\
 \leq & \|P_{N_0}e^{it\Delta} \phi\|_{L^{\infty}_{t}L^{2}_x(J\times \TTT^d)}\|R_1\|_{L^{\infty^-}_{t}L^{2}_x}\|P_{N_2} e^{it\Delta}\phi^{(2)}\|_{L^{\infty^-}_{t}L^2_x}\|P_{N_3} e^{it\Delta}\phi^{(3)}\|_{L^{1^+}_{t}L^{2}_x}\\
 \lesssim & |J|^{1-\epsilon} \frac{1}{N_1^{s_c-\alpha}} \|P_{N_0} \phi\|_{L^2_x}
 \|P_{N_2} \phi^{(2)}\|_{L^2_x}\|P_{N_3} \phi^{(3)}\|_{L^2_x}.
\end{align*}

So we obtain that 
\begin{equation}\label{eq:BdStep2}
\begin{split}
&\left|\int_{J\times\TTT^d} \overline{P_{N_0} e^{it\Delta}\phi} \mathcal{N}(\overline R_1, P_{N_2} e^{it\Delta}\phi^{(2)}, P_{N_3} e^{it\Delta}\phi^{(3)}) \,dxdt\right|\\
\lesssim& |J|^{1-\epsilon} \frac{N_2^\epsilon N_3^{\frac{d}{2}-\epsilon}}{N_1^{s_c-\alpha}} \|P_{N_0} \phi\|_{L^2_x}\|P_{N_2} \phi^{(2)}\|_{L^2_x}\|P_{N_3} \phi^{(3)}\|_{L^2_x}.
\end{split}
\end{equation}

\noindent\textbf{Step 3}
Average the estimates (\ref{eq:BdStep1}) (\ref{eq:BdStep2}) in Step 1 and Step 2, we obtain that
\begin{align}\label{eq:LocalTimeEstimate22}
\begin{split}
&\left|\int_{J\times\TTT^d} \overline{P_{N_0} e^{it\Delta}\phi} \mathcal{N}(\overline R_1, P_{N_2} e^{it\Delta}\phi^{(2)}, P_{N_3} e^{it\Delta}\phi^{(3)}) \,dxdt\right|\\
 \lesssim & |J|^{\frac{1}{2}} \frac{ N_2^{\frac{d-1}{8}} N_3^{\frac{d}{2}}}{N_1^{s_c + \frac{1}{4}-\alpha-\epsilon}} \|P_{N_0} \phi\|_{L^2_x}
 \|P_{N_2} \phi^{(2)}\|_{L^2_x}\|P_{N_3} \phi^{(3)}\|_{L^2_x}.
 \end{split}
\end{align}

By (\ref{eq:LocalTimeEstimate22}) and Lemma \ref{timelocalTransferPrinciple}, we hold that
\begin{equation}\label{eq:BdStep3}
\left|\int_0^\dd\int_{\mathbb{T}^d} \mathcal{N}(\overline{R}_{1},  D_2, D_3) \overline u_0 dxdt\right| \lesssim \dd^{\frac{1}{2}} \frac{N_1^\epsilon N_2^{\frac{d-1}{8}} N_3^{\frac{d}{2}}}{N_1^{s_c + \frac{1}{4}-\alpha}} \|u_0\|_{U_\Delta^2}\|D_2\|_{U_\Delta^2} \|D_3\|_{U_\Delta^2}.
\end{equation}

\noindent\textbf{Step 4 (only for $d=3, 4$)} If we set $\frac{1}{p_c^+} = \frac{d}{2(d+2)} - \epsilon$ and $\frac{1}{q} = 1-\frac{3}{p_c^+}$.
By H\"older inequality and Lemma \ref{LpLarge} after excluding a subset of probability $e^{-\frac{1}{\dd^c}}$, we have 
\begin{align*}
 &\left|\int_{J\times\TTT^d} \overline{u_0} \mathcal{N}_1
 (\overline R_1, D_2, D_3)\,dxdt\right|
 =\left|\int_{J\times\TTT^d} \overline{u_0}\ \overline R_1 D_2  D_3 \,dxdt\right|\\
 \leq & \|u_0\|_{L^{p_c^+}_{t,x}(J\times \TTT^d)}\|R_1\|_{L^{q}_{t,x}} \|D_2\|_{L^{p_c^+}_{t,x}} \|D_3\|_{L^{p_c^+}_{t,x}}\\
 \lesssim &  \frac{N_1^\epsilon }{N_1^{s_c-\alpha}} \|u_0\|_{U^{p_c^+}_\Delta}\|D_2\|_{U^{p_c^+}_\Delta}\|D_3\|_{U^{p_c^+}_\Delta}.
\end{align*}
And also we have
\begin{align*}
&\left|\int_{J\times\TTT^d} \overline{u_0} \mathcal{N}_2(\overline R_1, D_2, D_3) \,dxdt\right|
\lesssim  \left|\int_{J}\left(\int_{\TTT^d} \overline{u_0}\ \overline{R_1}\,dx\right) \left(\int_{\TTT^d} D_2 D_3\,dx\right)dt\right|\\
 \leq & \|u_0\|_{L^{p_c^+}_{t}L^{2}_x(J\times \TTT^d)}\|R_1\|_{L^{q}_{t}L^{2}_x}\|D_2\|_{L^{p_c^+}_{t}L^2_x}\|D_3\|_{L^{p_c^+}_{t}L^{2}_x}\\
 \lesssim &  \frac{N_1^\epsilon }{N_1^{s_c-\alpha}} \|u_0\|_{U^{p_c^+}_\Delta}\|D_2\|_{U^{p_c^+}_\Delta}\|D_3\|_{U^{p_c^+}_\Delta}.
\end{align*}

So we obtain that 
\begin{equation}\label{eq:BdStep4}
\left|\int_{J\times\TTT^d} \overline{u_0} \mathcal{N}(\overline R_1, D_2, D_3) \,dxdt\right|
\lesssim  \frac{N_1^\epsilon }{N_1^{s_c-\alpha}} \|u_0\|_{U^{p_c^+}_\Delta}\|D_2\|_{U^{p_c^+}_\Delta}\|D_3\|_{U^{p_c^+}_\Delta}.
\end{equation}

\noindent\textbf{Step 5}
 Finally, when $d=3, 4$, by the embedding properties {Remark} \ref{rmk:embedding} and (\ref{eq:Embedding}),
we average (\ref{eq:BdStep3}) and (\ref{eq:BdStep4}) with weights: $(\frac{8d-16}{5d-1}, \frac{15-3d}{5d-1})$; when $d\geq 5$, we directly use (\ref{eq:BdStep3}). Summarizing these two cases, we obtain that
\begin{equation}\label{eq:prop5.10}
\begin{split}
&\left|\int_0^\dd\int_{\mathbb{T}^d} \mathcal{N}(\overline{R}_{1},  D_2, D_3) \overline u_0 dxdt\right|\\
\lesssim &\dd^{4s_r(d)-\epsilon} \frac{ N_2^{s_c-s}N_3^{s_c-s}}{N_1^{s_c - s +s_r (d) + \frac{1}{4}-\alpha-\epsilon}}  \|u_0\|_{Y^{-s}}\|D_2\|_{X^s} \|D_3\|_{X^s}
\end{split}
\end{equation}
where \[s_r(d) = \begin{cases}
\frac{1}{7}& d= 3\\
\frac{4}{19}& d= 4\\
\frac{1}{4} & d\geq 5.
\end{cases}\]
Since $s<s_c + s_r (d) -\alpha$, we have the proposition.
\end{proof}

\begin{remark}
The proofs when the conjugate is on the different position are similar, for example, $\mathcal{N}(\overline{R_1}, D_2, D_3)$ in the case \textbf{B(d)}. The only difference between $\mathcal{N}(\overline{R_1}, D_2, D_3)$ and $\mathcal{N}({R_1}, \overline{D_2}, D_3)$ in \textbf{B(d)} is (\ref{eq:MatrixEstimate4}). In the $\mathcal{N}(\overline{R_1}, D_2, D_3)$, the set in (\ref{eq:MatrixEstimate4}) should be $\{ (n, n', n_2): n\neq n', 2\langle n_3-n, n+n_2\rangle = k,\ 
 2\langle n_3-n', n'+n_2\rangle = k\}$. First, we can count the number of pair $(n, n_2)$ and
by Lemma \ref{Counting1}, it is bounded by $N_1^{d-2+\epsilon} N_2^{d}$. And then the number of possible $n'$ is bounded $N_1^{d-2+\epsilon}$ by Lemma \ref{Counting2}. So the size of the upper lattice set is bounded by $N_1^{2(d-2)+\epsilon}N_2^{d}$, which is better than the corresponding bound in $\mathcal{N}(R_1, \overline{D_2}, D_3)$.
\end{remark}

\subsection*{Proof of Proposition \ref{NonlinearEstimate}}
\begin{proof}
Suppose dyadic coordinates $N_1\geq N_2 \geq N_3$, consider arbitrary 
${w_l^{(i)}}$ and ${u_l^{(0)}}$ satisfying ${w_l^{(i)}}(t) = w^{(i)}(t)$ and ${u_l^{(0)}}(t) = u^{(0)}(t)$, for $t\in I$ and $i = 1, 2, 3.$ Then we have

\begin{align*}
&\left|\int_0^\delta\int_{\mathbb{T}^4} \mathcal{N}(w^{(1)} + v_0^\omega, \overline{{w^{(2)}} + v_0^\omega}, w^{(3)} + v_0^\omega)\overline{u^{(0)}}dxdt\right|\\
= &\left|\int_0^\delta\int_{\mathbb{T}^4} \mathcal{N}({w_l^{(1)}} + v_0^\omega, \overline{{w_l^{(2)}} + v_0^\omega}, {w_l^{(3)}} + v_0^\omega)\overline{{u_l^{(0)}}}dxdt\right|\\
=& \sum_{N_0\lesssim N_1 \geq N_2 \geq N_3} \left|\int_0^\delta\int_{\mathbb{T}^4} \mathcal{N}(P_{N_1}({w_l^{(1)}} + v_0^\omega), \overline{P_{N_2}({w_l^{(2)}} + v_0^\omega)}, P_{N_3}({w_l^{(3)}} + v_0^\omega))\overline{P_{N_0}{u_l^{(0)}}}dxdt\right|
\end{align*}
There are only two cases: $N_0\sim N_1\geq N_2\geq N_3$ and $N_0\lesssim N_1 \sim N_2 \geq N_3$.

By Proposition 5.3 - 5.11, we can always sum up for these two cases to obtain the following estimate:

\begin{align*}
&\left|\int_0^\delta\int_{\mathbb{T}^d} \mathcal{N}(w^{(1)} + v_0^\omega, \overline{w^{(2)} + v_0^\omega}, w^{(3)} + v_0^\omega)\overline{u^{(0)}}\,dxdt\right|\\
\lesssim & \|{u_l^{(0)}}\|_{Y^{-s}}\left(\dd^{c\min\{1, s-s_c\}}\|{w_l^{(1)}}\|_{X^s}\|{w_l^{(2)}}\|_{X^s}\|{w_l^{(3)}}\|_{X^s}+\dd^{c}\sum_{J\subset\{1, 2, 3\}\atop J\neq \{1,2,3\}}\prod_{j\in J}\|{w_l^{(j)}}\|_{X^s}\right),
\end{align*}
by the definition of $X^s(I)$ and $Y^{-s}(I)$ in Definition \ref{def:LocalNorm}, we obtain Proposition \ref{NonlinearEstimate}.
\end{proof}

\section{Proof of the Theorem \ref{ASThm}}
To prove Theorem \ref{ASThm} (especially the case $s= s_c$), we should introduce two weaker norms $Z^s(I)$ and $Z'^{s}(I)$-norm than $X^s(I)$-norm.

\begin{definition}
  \[\|v\|_{Z^s(I)} := \sup_{J\subset I}
  \left( \sum_{N \in 2^\ZZZ} N^{4s+2-d} \|P_N v\|^4_{L^4(\TTT^d\times J)}
  \right)^{\frac{1}{4}} \text{ and } \|v\|_{Z'^{s} (I)} := \|v\|^{\frac{3}{4}}_{Z^s(I)}\|v\|^{\frac{1}{4}}_{X^s(I)}.\]
\end{definition}

The following property show us that $Z^s(I)$ is a weaker norm than $X^s(I)$.
\begin{prop}\label{prop:ZinX}
  \[\|v\|_{Z^s(I)} \lesssim \|v\|_{X^s(I)}\].
\end{prop}
\begin{proof}
  By the definition of $Z^s(I)$ and the following Strichartz type estimates
  (Proposition \ref{Strichartz}), we obtain that
  \[
  \sup_{J\subset I}
  \left( \sum_{N \text{ dydic number}} N^{4s+2-d} \|P_N v\|^4_{L^4(\TTT^d\times J)}  \right)^{\frac{1}{4}}
  \lesssim
  \left( \sum_{N \text{ dydic number}} N^{4s} \|P_N v\|^4_{U^4_\Delta}  \right)^{\frac{1}{4}}
  \lesssim \|v\|_{X^s(I)}.
  \]
\end{proof}

\begin{lem}[Bilinear estimates in \cite{herr2014strichartz}]\label{lem:bilinear}
  Assuming $|I|\leq 1$ and $N_1\geq N_2$, for any $v_1\in Y^0(I)$ and $v_2\in Y^{s_c}(I)$, where $s_c = \frac{d}{2}- 1$, there holds that
  \begin{equation}
    \|P_{N_1} v_1 P_{N_2} v_2\|_{L^2_{x,t}(\TTT^d\times I)}
    \lesssim (\frac{N_2}{N_1}+\frac{1}{N_2})^\kappa \|P_{N_1} v_1\|_{Y^0(I)}
    \|P_{N_2}v_2\|_{Y^{s_c}(I)}
  \end{equation}
  for some $\kappa >0$.
\end{lem}

\begin{remark}
  This Bilinear estimate is a simple d-dimension generalization of  Proposition 2.8 in \cite{herr2014strichartz}. The proof 
  of Lemma \ref{lem:bilinear} is almost the same as the $d=4$ case in \cite{herr2014strichartz} and   heavily rely on $L^p$ estimates in
  Proposition \ref{Strichartz} (for some $p<4$). In the proof not only
  we need the decoupling properties for spatial frequency, but also
  we need further trip partitions to apply the decoupling properties for time
  frequency.
\end{remark}

Let's introduce an refined nonlinear estimate for $s=s_c$ case, which is a d-dimension generalization of Lemma 3.2 in \cite{ionescu2012global2}.

\begin{prop}[Refined nonlinear estimate]\label{prop:refinednonlinear}
  For $v_k\in X^{s_c}(I)$, $k=1, 2, 3$, $|I|\leq 1$, we hold the estimate
  \begin{equation}\label{eq:nonlinear}
    \|\mathcal{I}(\mathcal{N}( \widetilde{v_1}, \widetilde{v_2}, \widetilde{v_3}))\|_{X^{s_c}(I)}\lesssim
    \sum_{\{i, j, k\}=\{1, 2, 3\}} \|v_i\|_{X^{s_c}(I)}\|v_j\|_{Z'^{s_c}(I)}\|v_k\|_{Z'^{s_c}(I)}
  \end{equation}
  where $\widetilde{v_k} =v_k$ or $\widetilde{v_k} = \overline{v_k}$
  for $k=1, 2, 3$.
\end{prop}

\begin{proof}
  By Proposition \ref{DualProp}, we
  suppose $N_0$, $N_1$, $N_2$, $N_3$ are dyadic, and by the symmetry, we assume $N_1\geq N_2 \geq N_3$. Since it's easy to check that $\mathcal{N}_2$ is simple to bound, we just need to show the case $\mathcal{N}_1$.
\begin{align*}
  &\|\mathcal{I}(\mathcal{N}_1( \widetilde{v_1}, \widetilde{v_2}, \widetilde{v_3}))\|_{X^{s_c}(I)} \lesssim
  \sup_{\|u_0\|_{Y^{-s_c}}=1} \left|\int_{\TTT^d\times I} \overline{u_0}
  \prod_{k=1}^3 \widetilde{v_k} \,dxdt\right|\\
  \leq & \sup_{\|u_0\|_{Y^{-s_c}}=1} \sum_{N_0, N_1\geq N_2\geq N_3}\left|
  \int_{\TTT^d\times I} \overline{P_{N_0} u_0}
  \prod_{k=1}^3 P_{N_k}\widetilde{v_k} \,dxdt\right|
\end{align*}

Then we know $N_1 \sim \max \{N_2, N_0\}$ by the spatial frequency orthogonality.
There are two cases:
\begin{enumerate}
  \item $N_0\sim N_1 \geq N_2 \geq N_3$;
  \item $N_0\leq N_2 \sim N_1 \geq N_3$.
\end{enumerate}

\case{1}{$N_0\sim N_1 \geq N_2 \geq N_3$}
By Cauchy-Schwartz inequality and Lemma \ref{lem:bilinear}, we have that
\begin{equation}\label{3.2}
\begin{split}
  & \left|\int \overline{P_{N_0} u_0} P_{N_1} \widetilde{v_1} P_{N_2} \widetilde{v_2}
  P_{N_3} \widetilde{v_3}\, dxdt\right|
  \leq \|P_{N_0} u_0 P_{N_2} v_2\|_{L^2_{x,t}}
    \|P_{N_1}v_1 P_{N_3} v_3\|_{L^2_{x,t}}\\
  \lesssim& (\frac{N_3}{N_1} + \frac{1}{N_3})^\kappa
  (\frac{N_2}{N_0} + \frac{1}{N_2})^\kappa \|P_{N_0} u_0\|_{Y^0(I)}
  \|P_{N_1}v_1\|_{Y^0(I)}\|P_{N_2}v_2\|_{X^{s_c}(I)}\|P_{N_3}v_3\|_{X^{s_c}(I)}
\end{split}
\end{equation}

Assume $\{C_j\}$ us a cube partition of size $N_2$, and $\{C_k\}$ is a cube
partition of size $N_3$. By $\{P_{C_j}P_{N_0} u_0 P_{N_2} v_2\}_j$
and $\{P_{C_k}P_{N_1}v_1 P_{N_3}v_3\}_k$ are both almost orthogonality,
Proposition \ref{Strichartz} and definition of $Z^{s_c}$ norm, we obtain that
\begin{equation}\label{3.3}
  \begin{split}
    & \left|\int \overline{P_{N_0} u_0} P_{N_1} \widetilde{v_1} P_{N_2} \widetilde{v_2}
    P_{N_3} \widetilde{v_3}\, dxdt\right|
    \leq \|P_{N_0} u_0 P_{N_2} v_2\|_{L^2_{x,t}}
    \|P_{N_1}v_1 P_{N_3} v_3\|_{L^2_{x,t}}\\
    \lesssim & (\sum_{C_j} \|P_{C_j}P_{N_0} u_0 P_{N_2} v_2\|^2_{L^2_{x,t}})^{\frac{1}{2}}
    (\sum_{C_k} \|P_{C_k}P_{N_1}v_1 P_{N_3}v_3 \|^2)^{\frac{1}{2}}\\
    \lesssim & (\sum_{C_j} \|P_{C_j}P_{N_0} u_0\|^2_{L^4_{x,t}}
    \|P_{N_2} u_2\|^2_{L^4_{x,t}})^{\frac{1}{2}}
    (\sum_{C_k} \|P_{C_k}P_{N_1}v_1\|_{L^4_{x,t}}^2
    \| P_{N_3}v_3 \|^2_{L^4_{x,t}})^{\frac{1}{2}}\\
    \lesssim & (\sum_{C_j} \|P_{C_j}P_{N_0} u_0\|^2_{Y^0(I)}
    (N_2^{\frac{d-2}{4}}\|P_{N_2} v_2\|_{L^4_{x,t}})^2)^{\frac{1}{2}}
    (\sum_{C_k} \|P_{C_k}P_{N_1} v_1\|^2_{Y^0(I)}
    (N_3^{\frac{d-2}{4}}\|P_{N_3} v_3\|_{L^4_{x,t}})^2)^{\frac{1}{2}}\\
    \lesssim& \|P_{N_0}u_0\|_{Y^0(I)}\|P_{N_1}v_1\|_{Y^0(I)}\|P_{N_2}v_2\|_{Z^{s_c}(I)}
    \|P_{N_2}v_2\|_{Z^{s_c}(I)}.
  \end{split}
\end{equation}

Interpolate (\ref{3.2}) and (\ref{3.3}), and $N_0\sim N_1$, we have
\begin{equation}\label{3.4}
\begin{split}
  &\left|\int \overline{P_{N_0} u_0} P_{N_1} \widetilde{v_1} P_{N_2} \widetilde{v_2}
  P_{N_3} \widetilde{v_3}\, dxdt\right|\\
  \lesssim&
  (\frac{N_3}{N_1} + \frac{1}{N_3})^{\kappa_1}
  (\frac{N_2}{N_0} + \frac{1}{N_2})^{\kappa_1}
  \|P_{N_0}u_0\|_{Y^{-s_c}(I)}\|P_{N_1}v_1\|_{X^{s_c}(I)}\|P_{N_2}v_2\|_{Z'^{s_c}(I)}
  \|P_{N_2}v_2\|_{Z'^{s_c}(I)}.
\end{split}
\end{equation}

Sum (\ref{3.4}) over all $N_0\sim N_1\geq N_2\geq N_3$,
\begin{align*}
    &\sum_{N_0\sim N_1\geq N_2\geq N_3}
    (\frac{N_3}{N_1} + \frac{1}{N_3})^{\kappa_1}
    (\frac{N_2}{N_0} + \frac{1}{N_2})^{\kappa_1}
    \|P_{N_0}u_0\|_{Y^{-s_c}(I)}\|P_{N_1}v_1\|_{X^{s_c}(I)}\|P_{N_2}v_2\|_{Z'^{s_c}(I)}
    \|P_{N_2}v_2\|_{Z'^{s_c}(I)}\\
    \lesssim & \|u_0\|_{Y^{-s_c}(I)}\|v_1\|_{X^{s_c}{I}}\|v_2\|_{Z'^{s_c}(I)}\|v_3\|_{Z'^{s_c}(I)}.
\end{align*}

\case{2}{$N_0\leq N_2\sim N_1\geq N_3$}
Similar, we have
\begin{equation}\label{3.5}
\begin{split}
  &\left|\int \overline{P_{N_0} u_0} P_{N_1} \widetilde{v_1} P_{N_2} \widetilde{v_2}
  P_{N_3} \widetilde{v_3}\, dxdt\right|\\
  \lesssim& (\frac{N_3}{N_1} + \frac{1}{N_3})^\kappa
  (\frac{N_0}{N_2} + \frac{1}{N_0})^\kappa \|P_{N_0} u_0\|_{Y^0(I)}
  \|P_{N_1}v_1\|_{Y^0(I)}\|P_{N_2}v_2\|_{X^{s_c}(I)}\|P_{N_3}u_3\|_{X^{s_c}(I)}.
\end{split}
\end{equation}

Similar with (\ref{3.3}), we obtain that:

\begin{equation}\label{3.6}
\begin{split}
  &\left|\int \overline{P_{N_0} u_0} P_{N_1} \widetilde{v_1} P_{N_2} \widetilde{v_2}
  P_{N_3} \widetilde{v_3}\, dxdt\right|\\
  \lesssim& \|P_{N_0} u_0\|_{Y^0(I)}
  \|P_{N_1}v_1\|_{Y^0(I)}\|P_{N_2}v_2\|_{Z^{s_c}(I)}\|P_{N_3}v_3\|_{Z^{s_c}(I)}.
\end{split}
\end{equation}

Interpolating (\ref{3.5}) and (\ref{3.6}), and summing over $N_0\leq N_2\sim N_1\geq N_3$,
we have
\begin{align*}
  &\sum_{N_0\leq N_2\sim N_1\geq N_3}
  \left|\int \overline{P_{N_0} u_0} P_{N_1} \widetilde{v_1} P_{N_2} \widetilde{v_2}
  P_{N_3} \widetilde{v_3}\, dxdt\right|\\
  \lesssim & \|P_{N_0} u_0\|_{Y^{-s_c}(I)}
  \|P_{N_1}v_1\|_{X^{s_c}(I)}\|P_{N_2}v_2\|_{Z'^{s_c}(I)}\|P_{N_3}v_3\|_{Z'^{s_c}(I)}.
\end{align*}

 Summarize these two cases, and similarly consider
 $N_1\geq N_3\geq N_2$, $N_2\geq N_1\geq N_3$, $N_2\geq N_3\geq N_1$,
 $N_3\geq N_1\geq N_2$, and $N_3\geq N_2\geq N_1$,
then we can get the desired estimate (\ref{eq:nonlinear}).

\end{proof}

\begin{proof}[Proof of Theorem \ref{ASThm}]
Suppose $d\geq 3$,  $s_r(d)$ is defined as (\ref{coef:srd})
and $0\leq \alpha< s_r(d)$.  

Consider the mapping 
\[
\Phi (w) = \mathcal{I}(\mathcal{N}(w + v_0^\omega)),
\]
and then the fixed point $w = \Phi(w)$ of the mapping $\Phi$ is the solution of IVP (\ref{GNLS.IVP}).

\case{1}{$s_c<s < s_c + s_r(d)-\alpha$}
Consider the set 
\[
S = \{w\in X^{s}(I): \|w\|_{X^s(I)}\leq 1\}.
\]
where $I = [0,\delta]$ and $\delta$ is to be determined.

To show $\Phi$ is a contraction mapping in $S$. Using Proposition \ref{DualProp}, Proposition \ref{NonlinearEstimate} and choosing $\delta$ small enough, we obtain that
\begin{align*}
\|\Phi(w)\|_{X^s(I)}&\lesssim  \delta^{c\min \{1, s-s_c\}} (1+\|w\|_{X^s(I)}+ \|w\|^2_{X^s(I)} +\|w\|^3_{X^s(I)})\leq 1.
\end{align*}
For any $w, v \in S$, using Proposition \ref{DualProp}, Proposition \ref{NonlinearEstimate} and choosing $\delta$ small enough, there exists $0<k<1$ such that
\begin{align*}
\|\Phi(w) - \Phi(v)\|_{X^s(I)}&\lesssim  \delta^{c\min \{1, s-s_c\}} (1+ \|w\|_{X^s(I)}+\|v\|_{X^s(I)}+\|w\|^2_{X^s(I)}+\|v\|^2_{X^s(I)}) \|w-v\|_{X^s(I)}\\
&\leq k\|w-v\|_{X^s(I)}.
\end{align*}
So $\Phi$ is a contraction mapping.

\case{2}{$s = s_c$}
Consider the set
\[
S = \{w\in X^{s_c}(I): \|w\|_{X^{s_c}(I)}\leq 1, \quad \|w\|_{Z'^{s_c}(I)}\leq a\}.
\]
where $I = [0,\delta]$. $a$ and $\delta$ is to be determined.

To show $\Phi$ is a contraction mapping in $S$. By Proposition \ref{DualProp}, Proposition \ref{NonlinearEstimate}, Proposition \ref{prop:refinednonlinear}, choosing $\delta$ small enough, we obtain that
\begin{align*}
\|\Phi(w)\|_{X^{s_c}(I)}&\lesssim  \delta^{c} (1+\|w\|_{X^{s_c}(I)}+ \|w\|^2_{X^{s_c}(I)}) +
\|w\|^2_{Z'^{s_c}(I)} \|w\|_{X^{s_c}(I)}\\
&\lesssim \delta^c + a^2.
\end{align*}
and also 
\begin{align*}
\|\Phi(w)\|_{Z'^{s_c}(I)}&\lesssim  \delta^{c} (1+\|w\|_{X^{s_c}(I)}+ \|w\|^2_{X^{s_c}(I)}) +
\|w\|^2_{Z'^{s_c}(I)} \|w\|_{X^{s_c}(I)}\\
&\lesssim \delta^c + a^2.
\end{align*}

For any $w, v \in S$, by Proposition \ref{DualProp}, Proposition \ref{NonlinearEstimate} and Proposition \ref{prop:refinednonlinear}, choosing $\delta$ small enough, there exists $0<k<1$ such that
\begin{align*}
\|\Phi(w) - \Phi(v)\|_{X^{s_c}(I)}&\lesssim (\|w\|_{Z'^{s_c}(I)}+ \|v\|_{Z'^{s_c}(I)}) (\|w\|_{X^{s_c}(I)(I)}+ \|v\|_{X^{s_c}(I)})\|w-v\|_{X^{s_c}(I)}\\
&+ \delta^c (\|w\|_{X^{s_c}(I)}+ \|v\|_{X^{s_c}(I)}+1)\|w-v\|_{X^{s_c}(I)}\\
&\lesssim (a + \delta^c) \|w-v\|_{X^{s_c}(I)}.
\end{align*}

Set $a = \delta$ and let $\delta$ small enough, then we obtain that $\Phi$ is a contraction mapping.

\end{proof}

\section{The analog result in $X^{s,b}$ space}
 $X^{s, b}$ spaces (also known as \textit{Fourier restriction spaces} or \textit{Bourgain spaces}) were firstly introduced by Bourgain \cite{bourgain1993fourier}\cite{bourgain1993fourier2} in the context of Schr\"odinger and KdV equations. A nice summary is give by Tao (Section 2.6 in \cite{book:tao}).

\begin{thm}[Analog of Theorem \ref{MainThm} in $X^{s, b}$] \label{thm:MainThmXsb}
Suppose $d\geq 3$ and   
$s_r(d)$ is defined by (\ref{coef:srd})
Let $0\leq \alpha< s_r(d)$, $s\in (s_c, s_c + s_r(d)-\alpha)$.  Then there exist some $b>\frac{1}{2}$, $\delta_0 > 0$ and $r = r(s, \alpha) > 0$ such that for any $ 0 < \delta < \delta_0$, there exists $\Omega_\delta \in A$ with
$$
\mathbb{P}(\Omega_\delta^c) < e^{-\frac{1}{\delta^r}},
$$
and for each $\omega \in \Omega_\delta$ there exists a unique solution $u$ of (\ref{NLS.IVP}) in the space
$$
S(t) \phi^{\omega} + X^{s, b}([0, \delta])_{\text{dist}},
$$
where $S(t)\phi^{\omega}$ is the linear evolution of the initial data $\phi^{\omega}$ given by (\ref{RandomInitialData}).
\end{thm}

\begin{remark}
The proof of Theorem 7.1 follows the analog nonlinear estimate in $X^{s, b}$ space (Proposition \ref{NonlinearEstimateXsb}). In the proof of Proposition \ref{NonlinearEstimateXsb}, we can see the reason why $s=s_c$ case fails in $X^{s, b}$ space.
\end{remark}

\subsection{$X^{s,b}$ space and some properties}
Let's recall the definition and some properties of $X^{s, b}$ spaces.
\begin{definition}
Suppose $d\in \ZZZ^+$ and $s, b\in \RRR$, for any $u: \RRR\times\TTT^d\to \mathbb{C}$, $u\in X^{s,d}(\RRR\times\TTT^d)$ (short for $X^{s,b}$) if 
\[
\|u\|_{X^{s,b}} := \|\langle n\rangle^s \langle \lambda+|n|^2 \rangle^b \,\widehat{u}(n,\lambda)\|_{l^2_n(\ZZZ^d)L^2_\lambda(\RRR)}< +\infty,
\]
where $\widehat{u}(n,\lambda)$ is the space-time Fourier transformation of $u$.
Note that $\|u\|_{X^{s,b}}$ can be also defined as $\|e^{-it\Delta}u\|_{H^b_t H_x^s(\RRR\times\TTT^d)}$. 
\end{definition}

\begin{definition}[The corresponding restriction spaces to a time interval $I$]
Suppose $d\in \ZZZ^+$, $s, b\in \RRR$ and $I$ is a time interval in $\RRR$, for any $u: I\times\TTT^d\to \mathbb{C}$, then $u\in X^{s,b}(I\times\TTT^d)$ (short for $X^{s,b}(I)$) if
\[
\|u\|_{X^{s,b}(I)}:= \inf_{v\in X^{s,b}} \{\|v\|_{X^{s,b}} : v(t)=u(t) \text{ for all } t\in I\}<+\infty.
\]
\end{definition}

\begin{remark}[Some embedding properties]
For $s\leq s'$ and $b\leq b'$, we obtain that 
\[
X^{s,b}\hookrightarrow  X^{s', b'}.
\]
Furthermore, if $b> \frac{1}{2}$, then 
\[
X^{s,b} \hookrightarrow L^\infty_t(\RRR, H^s(\TTT^d)).
\]
\end{remark}

\begin{prop}[Analog of Proposition \ref{TransferPrinciple} in $X^{s,b}$]\label{TransferPrincipleXsb}
Let $T_0 : L_x^2 \times \cdots \times L_x^2 \to L^{1}_{x, loc}(\TTT^d)$ be an m-linear operator. Assume that for some $1\leq p \leq \infty$
\begin{equation}\label{eq:7.1}
\| T_0 (e^{it\Delta }\phi_1,\cdots, e^{it\Delta }\phi_m)\|_{L^p(\mathbb{R}\times\mathbb{T}^d)} \lesssim \prod_{i=1}^m\|\phi_i\|_{L^2(\mathbb{T}^d)}.
\end{equation}
Then, for any $b>\frac{1}{2}$, there exists an extension $T : X^{0, b}\times \cdots \times X^{0, b} \to L^p(\mathbb{R}\times\mathbb{T}^d)$ satisfying
\begin{equation}
\|T(u_1, \cdots, u_m)\|_{L^p(\mathbb{R}\times\mathbb{T}^d)} \lesssim \prod_{i=1}^m \|u_i\|_{X^{0, b}};
\end{equation}
and  such that $T(u_1, \cdots , u_m) (t, \cdot) = T_0 (u_1(t), \cdots , u_m(t))(\cdot)$, a.e.
\end{prop}
\begin{proof}
Suppose that for all $i = 1, \cdots, m$,
\begin{equation}
\label{eq:fourierdecomposition}
u_i (x,t) = \int_\RRR d{\lambda_i}\,\sum_{n_i\in \ZZZ^d}  \widehat{u_i} (n_i, \lambda_i) e^{in_i\cdot x + \lambda_i t}
=\int_\RRR e^{it\Delta} \phi^{(i)}_{\mu_i} d{\mu_i},
\end{equation}
where $\phi^{(i)}_{\mu_i} : = \sum_{n_i\in \ZZZ^d}  \widehat{u_i} (n_i, \mu_i - |n_i|^2)e^{i\mu_it} e^{in_i\cdot x}$ and $\mu_i = \lambda_i + |n_i|^2$.

Then, by (\ref{eq:fourierdecomposition}), Minkowski integral inequality and (\ref{eq:7.1}), we obtain that
\begin{equation}\label{eq:goal}
\begin{split}
&\|T(u_1, \cdots, u_m)\|_{L^p(\mathbb{R}\times\mathbb{T}^d)}
= \|T(\int_\RRR e^{it\Delta} \phi^{(1)}_{\mu_1} d{\mu_1}, \cdots, \int_\RRR e^{it\Delta} \phi^{(m)}_{\mu_m} d{\mu_m})\|_{L^p(\mathbb{R}\times\mathbb{T}^d)}\\
\leq & \int_{\RRR^m}\prod_{i=1}^{m} d\mu_i 
\|T_0(e^{it\Delta} \phi^{(1)}_{\mu_1}, \cdots,  e^{it\Delta} \phi^{(m)}_{\mu_m})\|_{L^p(\mathbb{R}\times\mathbb{T}^d)}\\
\lesssim&\prod_{i=1}^{m}  \int_{\RRR} \|\phi_{\mu_i}^{(i)}\|_{L^2_x(\TTT^d)} d\mu_i.
\end{split}
\end{equation}
For a fixed $i$ and any $b>\frac{1}{2}$, by H\"older inequality and 
$\int_\RRR \frac{1}{\langle \mu_i\rangle^{2b}}\,d\mu < +\infty$, we have that 
\begin{equation}\label{eq:X0b}
\begin{split}
\int_{\RRR} \|\phi_{\mu_i}^{(i)}\|_{L^2_x(\TTT^d)} d\mu_i &=
\int_{\RRR} \frac{1}{\langle \mu_i \rangle^b} \|\langle\mu_i\rangle^b \phi_{\mu_i}^{(i)}\|_{L^2_x(\TTT^d)} d\mu_i \\
&\leq \left(\int_{\RRR} \frac{1}{\langle \mu_i \rangle^{2b}} d\mu_i\right)^{\frac{1}{2}}\left(\int_{\RRR} \|\langle\mu_i\rangle^b \phi_{\mu_i}^{(i)}\|^2_{L^2_x(\TTT^d)} d\mu_i\right)^{\frac{1}{2}}\\
&=\left(\int_{\RRR} \sum_{n_i} \langle \lambda_i +|n_i|^2 \rangle^{2b} |\widehat{u_i}(n_i, \lambda_i)|^2 d\lambda\right)^{\frac{1}{2}}\\
&= \|u_i\|_{X^{0, b}}.
\end{split}
\end{equation}

By (\ref{eq:goal}) and (\ref{eq:X0b}), we obtain the proposition.
\end{proof}

Following Proposition \ref{Strichartz} and Proposition \ref{TransferPrincipleXsb}, we obtain the following corollary.
\begin{cor}[Analog of Proposition \ref{Strichartz}]\label{StrichartzXsb}
Let $b>\frac{1}{2}$ and $p>p_c$, where $p_c = \frac{2(d+2)}{d}$. For all $N\geq 1$ we have
\begin{eqnarray}
\|P_{N} u \|_{L^p_{x,t}(\mathbb{T}\times\mathbb{T}^d)}&\lesssim& N^{\frac{d}{2}-\frac{d+2}{p}}\|P_{C} u \|_{X^{0, b}},\\
\|P_{C} u \|_{L^p_{x,t}(\mathbb{T}\times\mathbb{T}^d)}&\lesssim& N^{\frac{d}{2}-\frac{d+2}{p}}\|P_{C} u \|_{X^{0, b}},
\end{eqnarray}
where $C$ is a cube in $\mathbb{Z}^d$ with sides parallel to the axis of side length $N$.
\end{cor}

\begin{prop}[$X^{s,b}$ norm of Duhamel's formula]
For any $s>\frac{1}{2}$, $s\in \RRR$, and $I$ is a time interval $[0, \dd]$, we obtain that 
\begin{equation}
\| \mathcal{I}(f) \|_{X^{s, b}(I)}\leq \sup_{\|v\|_{\widetilde{X}^{-s, 1-b}(I)}=1}\left|\int_0^\delta
\int_{\mathbb{T}^d} f(t, x)\overline{v(t,x)}dxdt\right|.
\end{equation}
where $\|v\|_{\widetilde{X}^{-s, 1-b}(I)} := \|e^{it\Delta}v\|_{H^{1-b}_tH^{-s}_x(I\times\TTT^d)}$
\end{prop}

\subsection{Nonlinear estimate in $X^{s, b}$}

\begin{prop}[Analog of Proposition \ref{NonlinearEstimate} in $X^{s, b}$]\label{NonlinearEstimateXsb}
Suppose $d\geq 3$ and  $s_r(d)$ is given in (\ref{coef:srd}).
Let $0\leq \alpha< s_r'(d)$, $s\in (s_c, s_c + s_r'(d)-\alpha)$, $r>0$, $0\leq \dd<1$, and $I = [0,\dd]$.    There exist some $b>\frac{1}{2}$, $\Omega_\dd\subset \Omega$ with $\PPP(\Omega_\dd^c)<e^{-1/\dd^r}$ and $c>0$, such that we obtain that
\begin{align*}
&\left|\int_0^\delta\int_{\mathbb{T}^d} \mathcal{N}(w^{(1)} + v_0^\omega, \overline{w^{(2)} + v_0^\omega}, w^{(3)} + v_0^\omega)\overline{u^{(0)}}\,dxdt\right|\\
\lesssim & \|u^{(0)}\|_{\widetilde{X}^{-s, 1-b}(I)}\left(\dd^{c\min{\{1, s-s_c\}}}\|w^{(1)}\|_{X^{s, b}(I)}\|w^{(2)}\|_{X^{s, b}(I)}\|w^{(3)}\|_{X^{s, b}(I)}+\dd^{c}\sum_{S_J\subset\{1, 2, 3\}\atop J\neq \{1,2,3\}}\prod_{j\in S_J}\|w^{(j)}\|_{X^{s, b}(I)}\right),
\end{align*}
where $v_0^\omega$ is defined (\ref{def:v0}), $u^{(0)}\in \widetilde{X}^{-s, 1-b}(I)$
 and $w^{(i)}\in X^{s, b}(I)$ for $i=1, 2, 3$. 
 (when the subset $S_J = \emptyset$, $\prod_{j\in S_J}\|w^{(j)}\|_{X^{s, b}(I)} =1.$)
\end{prop}

\begin{proof}
The proof of Proposition \ref{NonlinearEstimateXsb} is similar with the proof of Proposition \ref{NonlinearEstimate}: we first dyadically decompose the terms in each position of the nonlinear term $\mathcal{N}(u_1, u_2, u_3)$, and then we have the same cases:
Denote \begin{equation}
R_i = P_{N_i} v_0^\omega \text{ and }  D_i = P_{N_i} w \text{ for } i\in\{1, 2, 3\}.
\end{equation} 
The list of all cases of $(u_1, u_2, u_3)$ is below:

\begin{enumerate}
\item[A.]  $u_1 = D_1$: 
\begin{enumerate}
\item $(D_1, D_2, D_3)$;
\item $(D_1, D_2, R_3)$;
\item $(D_1, R_2, D_3)$;
\item $(D_1, R_2, R_3)$;
\end{enumerate}
\item[B.] $u_1 = R_1$:
\begin{enumerate}
\item $(R_1, R_2, R_3)$;
\item $(R_1, R_2, D_3)$;
\item $(R_1, D_2, R_3)$;
\item $(R_1, D_2, D_3)$.
\end{enumerate}
\end{enumerate}

For the \textbf{Case A}, we can control the nonlinear terms as Proposition 5.3 - 5.6 in a similar approach. For simplicity, let me just show the \textbf{Case A (a)} as an example. 

Choose $b = \frac{1}{2}+$ which is close to $\frac{1}{2}$ enough. We follow the proof of Proposition 5.3 almost identically, but Proposition \ref{StrichartzXsb} take place of Proposition \ref{Strichartz}. Then we obtain that
\begin{equation}\label{eq:XsbAa1}
\begin{split}
&\left|\int_0^\dd\int_{\mathbb{T}^d} \mathcal{N}(\widetilde{D}_{1}, \widetilde D_2, \widetilde D_3) \overline u_0 \, dxdt\right|\\ \lesssim& \dd^{c \min\{1, s-s_c\}} (\frac{N_3\min{\{N_0, N_2\}}}{N_2^2})^c \frac{1}{N_2^{s-s_c}} \|{u}_0\|_{X^{-s, b}}\|{D}_1\|_{X^{s, b}}\|{D}_2\|_{X^{s, b}}\|{D}_3\|_{X^{s, b}}.
\end{split}
\end{equation}

To get the $\|{u}_0\|_{X^{-s, 1-b}}$ instead of $\|{u}_0\|_{X^{-s, b}}$, we need another nonlinear estimate. By the same cube decomposition, H\"{o}lder inequality, and Proposition \ref{StrichartzXsb}, we obtain that (here we only consider the main part of $\mathcal{N}(\widetilde{D_1}, \widetilde D_2, \widetilde D_3)$, and the remaining part is easily bounded)
\begin{equation}\label{eq:XsbAa2}
\begin{split}
& \left|\int_0^\dd\int_{\mathbb{T}^d}\overline{u}_0\widetilde{D}_1\widetilde{D}_2\widetilde{D}_3 \, dxdt\right|\\
 \leq& \sum_{C_j\sim C_k}\int_0^\dd\int_{\mathbb{T}^d}\left|(P_{C_j}\overline{u}_0)(P_{C_k}\widetilde{D}_1)\widetilde{D}_2\widetilde{D}_3\right| \, dxdt\\
 \lesssim&  \sum_{C_j\sim C_k}  \|P_{C_j}u_0\|_{L_{t,x}^{2}}\|P_{C_k}D_1\|_{L_{t,x}^{4}}\|D_2\|_{L_{t,x}^{4}}\|D_3\|_{L_{t,x}^{\infty}}\\
 \lesssim & \sum_{C_j\sim C_k} {N_2^{s_c}N_3^{s_c+1}} \|P_{C_j}u_0\|_{X^{0, 0}}\|P_{C_k}D_1\|_{X^{0, b}}\|D_2\|_{X^{0, b}}\|D_3\|_{X^{0, b}}\\
 \lesssim& \frac{N_3^{1-(s-s_c)}}{N_2^{s-s_c}} \|u_0\|_{X^{-s, 0}}\|D_1\|_{X^{s, b}}\|D_2\|_{X^{s_c, b}}\|D_3\|_{X^{s, b}}
\end{split}
\end{equation}

Using complex interpolation method from $X^{s, b}$ and $X^{s, 0}$ to $X^{s, 1-b}$ and interpolating (\ref{eq:XsbAa1}) and (\ref{eq:XsbAa2})  (actually we don't interpolate (\ref{eq:XsbAa1}) and (\ref{eq:XsbAa2}) directly but interpolate two estimates in the process of (\ref{eq:XsbAa1}) and (\ref{eq:XsbAa2})), we obtain that 

\begin{equation}\label{eq:XsbAa3}
\begin{split}
&\left|\int_0^\dd\int_{\mathbb{T}^d} \mathcal{N}(\widetilde{D}_{1}, \widetilde D_2, \widetilde D_3) \overline u_0 \, dxdt\right|\\ \lesssim& \dd^{c \min\{1, s-s_c\}} (\frac{N_3\min{\{N_0, N_2\}}}{N_2^2})^{c-\epsilon} \frac{N_3^\epsilon}{N_2^{s-s_c}} \|{u}_0\|_{X^{-s, 1-b}}\|{D}_1\|_{X^{s, b}}\|{D}_2\|_{X^{s, b}}\|{D}_3\|_{X^{s, b}}.
\end{split}
\end{equation}

Observe the bound in (\ref{eq:XsbAa3}), to sum up over $N_2$ and $N_3$, we need $s>s_c + \epsilon$, which is the reason why we can't obtain $s=s_c$ case in $X^{s, b}$ space.

For the \textbf{Case B}, we could obtain analogs of Proposition 5.7-5.10 by modifying the proofs a little bit. Let me show the \textbf{Case B (d)} as an example. Similar with the proof of Proposition 5.10, we only focus $N_0\sim N_1\geq N_2\geq N_3$. Then
\begin{equation}\label{eq:XsbBd1}
\left|\int_0^\dd\int_{\mathbb{T}^d} \mathcal{N}(\widetilde{R}_{1}, \widetilde D_2, \widetilde D_3) \overline u_0 dxdt\right|\leq \|R_1\|_{L^\infty_{t, x}}\|D_2\|_{L^3_{t, x}}\|D_3\|_{L^3_{t, x}}\|u_0\|_{L^3_{t, x}}.
\end{equation}

By Hausdorff-Young inequality w.r.p.t the time $t$ and H\"older inequality, we obtain that for a general function $u$ and dyadic number $N$,
\begin{equation}\label{eq:XsbBd2}
\begin{split}
\|P_N u\|_{L^3_{t, x}} &\lesssim \sum_{|n|\sim N}\|e^{ix\cdot n}\int \widehat{u}(n, \lambda) e^{\lambda t}\, d\lambda\|_{L^3_{t, x}}\\
&\lesssim \sum_{|n|\sim N} \left( \int |\widehat{u}(n, \lambda)|^{\frac{3}{2}} \, d\lambda \right)^{\frac{2}{3}}\lesssim \|P_N u\|_{X^{\epsilon, \frac{1}{6}+\epsilon}}.
\end{split}
\end{equation}

By Corollary \ref{CorLp} and (\ref{eq:XsbBd2}), we obtain that

\begin{equation}\label{eq:XsbBd3}
\begin{split}
\text{LHS of }(\ref{eq:XsbBd1})&\lesssim \frac{\log N_1}{N_1^{s_c-\alpha}}\|D_2\|_{X^{\epsilon, \frac{1}{6}+\epsilon}}
\|D_3\|_{X^{\epsilon, \frac{1}{6}+\epsilon}}
\|u_0\|_{X^{\epsilon, \frac{1}{6}+\epsilon}}\\
&\lesssim \frac{N_1^{s-s_c +\alpha +\epsilon}}{N_2^s N_3^s}\|D_2\|_{X^{s, \frac{1}{6}+\epsilon}}
\|D_3\|_{X^{s, \frac{1}{6}+\epsilon}}
\|u_0\|_{X^{-s, \frac{1}{6}+\epsilon}}
\end{split}
\end{equation}

Thus the estimate (\ref{eq:XsbBd3}) is conclusive provided for some deterministic term,  we consider the contribution $\widehat{u_0}|_{\langle \lambda + |n|^2\rangle> N_1^{4(s-s_c+\alpha)}}$.  Thus in this case, LHS of (\ref{eq:XsbBd1}) mab be estimated assuming
\[
\langle \lambda + |n|^2\rangle\leq N_1^{4(s-s_c+\alpha)}.
\]

For $s-s_c+\alpha <\frac{1}{2}$, replacing Proposition \ref{TransferPrinciple} and \ref{timelocalTransferPrinciple} by Proposition \ref{TransferPrincipleXsb}, we could recover the proof of Proposition 5.10 and obtain an analog of (\ref{eq:prop5.10}),
\begin{equation}\label{eq:XsbBd4}
\begin{split}
&\left|\int_0^\dd\int_{\mathbb{T}^d} \mathcal{N}(\overline{R}_{1},  D_2, D_3) \overline u_0 dxdt\right|\\
\lesssim &\dd^{4s_r(d)-\epsilon} \frac{ N_2^{s_c-s}N_3^{s_c-s}}{N_1^{s_c - s +s_r (d) + \frac{1}{4}-\alpha-\epsilon}}  \|u_0\|_{\widetilde{X}^{-s, b}}\|D_2\|_{X^{s, b}} \|D_3\|_{X^{s, b}}\\
\lesssim &\dd^{4s_r(d)-\epsilon} \frac{  N_1^{4(s-s_c+\alpha)(2b-1)}N_2^{s_c-s}N_3^{s_c-s}}{N_1^{s_c - s +s_r (d) + \frac{1}{4}-\alpha-\epsilon}}  \|u_0\|_{\widetilde{X}^{-s, 1-b}}\|D_2\|_{X^{s, b}} \|D_3\|_{X^{s, b}}.
\end{split}
\end{equation}
where $s_r(d)$ is defined by (\ref{coef:srd}), since $s<s_c+s_r(d)-\alpha$, we have the proposition.

In a similar idea, we can also recover the other cases in $X^{s, b}$.
\end{proof}

\bibliographystyle{amsplain}
\bibliography{Yue_ref}{}

\end{document}